\documentclass[11pt]{amsart}
\usepackage{color}
\usepackage[english]{babel}
\usepackage[latin1]{inputenc}
\usepackage{float}
\usepackage{latexsym}
\usepackage{amsmath}
\usepackage{amssymb}
\usepackage{graphicx}
\usepackage{amsthm}
\usepackage{amsfonts}

\usepackage[T1]{fontenc}

\usepackage{pdflscape}

\newtheorem{thm}{Theorem}[section]
\newtheorem{prop}[thm]{Proposition}
\newtheorem{cor}[thm]{Corollary}

\newtheorem{lem}[thm]{Lemma}
\newtheorem{defn}[thm]{Definition}
\newtheorem{example}[thm]{Example}
\newtheorem{remark}[thm]{Remark}

\newcommand{\R}{\mathbb{R}}
\newcommand{\RR}{\mathbb{R}}

\newcommand{\C}{\mathbb{C}}
\newcommand{\NN}{\mathbb{N}}
\newcommand{\N}{\mathbb{N}}

\newcommand{\PP}{\mathbb{P}}
\newcommand{\K}{\mathbb{K}}
\newcommand{\Le}{\mathbb{Le}}

\newcommand{\Lc}{\mathcal{L}} 
\newcommand{\Nc}{\mathcal{N}} 
\newcommand{\Ic}{\mathcal{I}} 
\newcommand{\Pc}{\mathcal{P}} 
\newcommand{\Sc}{\mathcal{S}} 
\newcommand{\Qc}{\mathcal{Q}} 
\newcommand{\Vc}{\mathcal{V}}

\newcommand{\Mon}{\mathcal{M}}
\newcommand{\Span}[1]{\<{#1}\>}
\newcommand{\SpanDeg}[2]{#1_{\<{#2}\>}}
\newcommand{\clK}{\overline{\mathbb K}}
\newcommand{\xx}{\mathbf x}
\newcommand{\uu}{\mathbf u}
\newcommand{\vv}{\mathbf v}
\newcommand{\gb}{\mathbf g} 
 
\newcommand{\yy}{\mathbf y} 
\newcommand{\zz}{\mathbf z}
\newcommand{\s}{\mathbf s}

\newcommand{\rank}{\mathrm{rank}}

\def\<{\langle}
\def\>{\rangle}
\DeclareMathOperator\re{Re}
\DeclareMathOperator\im{Im}

\begin{document}

\title{Exact relaxation for polynomial optimization on semi-algebraic sets}

\author{Marta Abril Bucero, Bernard Mourrain}
\address{Marta Abril Bucero, Bernard Mourrain: Galaad, Inria
  M\'editerran\'ee, 06902 Sophia Antipolis}

\begin{abstract}
In this paper, we study the problem of computing the infimum  
of a real polynomial function $f$ on a closed basic semialgebraic set
$S$ and the points where this infimum is reached, if they exist.
We show that when the infimum is reached, a Semi-Definite Program 
hierarchy constructed from the Karush-Kuhn-Tucker ideal is always exact and 
that the vanishing ideal of the KKT minimizer points is generated by
the kernel of the associated moment matrix in that degree, even if
this ideal is not zero-dimensional. We also show that this relaxation
allows to detect when there is no KKT minimizer. Analysing the
properties of the Fritz John variety, we show how to find all the minimizers of $f$.
We prove that the exactness of the relaxation depends only on the real points which satisfy these constraints.
This exploits representations of positive polynomials as elements
of the preordering modulo the KKT ideal, which only involves
polynomials in the initial set of variables.
The approach provides a uniform treatment of different optimization
problems considered previously.
Applications to global optimization, optimization on semialgebraic 
sets defined by regular sets of constraints, optimization on finite 
semialgebraic sets and real radical computation are given. 
\end{abstract}

\maketitle
\setcounter{page}{1}

\section{Introduction}


The problem we consider in this paper is the following: 
\begin{eqnarray}\label{problem1}
 \inf_{\xx \in \R^n}&& f(\xx) \\ 
 s.t.&& g_1^{0}(\xx)=\cdots=g_{n_1}^{0}(\xx)=0 \nonumber \\
     &&g_1^{+}(\xx)\ge 0,\ldots, g_{n_2}^{+}(\xx)\ge 0  \nonumber
\end{eqnarray}
where $f, g_1^{0}, \ldots, g_{n_{1}}^{0}$, $g_1^{+}, \ldots, g_{n_{2}}^{+}
\in \R[\xx]$ are polynomial functions in $n$ variables $x_{1},
\ldots, x_{n}$.

Hereafter, we fix the set of constraints
$\gb=\{g_1^{0},\ldots,g_{n_1}^{0}; g_1^{+}, \ldots, g_{n_{2}}^{+} \}=\{\gb^0;\gb^+\}$
and denoted by $S$ the basic semi-algebraic
set defined by these constraints.


The points $\xx^{*}\in S$ which satisfy $f 
(\xx^{*}) = \inf_{\xx \in S} f (\xx)$ are called the
{\em minimizers points} of $f$ on $S$.
If the set of minimizers is not empty, we say that the {\em  minimization problem is feasible}.

The objectives of the method we consider are to detect if the minimization problem is feasible
and to compute the minimum value of $f$ and the minimizer points where this minimum value is
reached, when the problem is feasible.
Though this global minimization problem is known to be NP-hard (see
e.g. \cite{Nesterov2000}),
a practical challenge is to devise methods which can approximate or
compute efficiently the solutions of the problem.

About a decade ago, a relaxation approach has been proposed in
\cite{Las01} (see also \cite{Par03}, \cite{Shor87}) to solve this
difficult problem. 
Instead of searching points where the polynomial $f$ reaches its
minimum $f^{*}$, a probability measure
which minimizes the function $f$ is searched. This problem is relaxed into a
hierarchy of finite dimensional convex minimization problems, which can
be solved by Semi-Definite Programming (SDP) techniques. The sequence
of SDP minima converges to the minimum $f^{*}$ \cite{Las01}. This hierarchy of
SDP problems can be formulated in terms of linear matrix inequalities
on moment matrices associated to the set of monomials of degree $t$ or
less, for increasing values of $t$.  The dual hierarchy can be
described as a sequence of maximization problems over the cone of
polynomials that are Sums of Squares (SoS). A feasibility condition is
needed to prove that this dual hierarchy of maximization problems also
converges to the minimum $f^{*}$, i.e. that there is no duality gap.

This approach provides a very interesting way to approximate a global
optimum of a polynomial function on $S$. But one may wonder if using
this approach, it
is possible to compute in a finite number of steps, this minimum and the minimizer
points when the problem is feasible.
From a computational point of view, the following issues need to be addressed:
\begin{enumerate}
 \item Is it possible to use an {\em exact} SDP hierarchy, i.e. which 
   converges in a finite number of steps?
 \item How can we recover all the points where the optimum is achieved
   if the optimization problem is feasible?
\end{enumerate}
To address the first issue, the following strategy has been considered:
add polynomial inequalities or equalities satisfied by the points where
the function $f$ is minimum. 

A first family of methods are used when the set $S$ is compact or when
the minimizer set can be bounded easily. By adding an inequality
constraint, one can then transform $S$ into a compact subset of
$\R^{n}$, for which exact hierarchies can be used \cite{Las01},
\cite{Marshall03}. It is shown in \cite{Lau07} that
if the complex variety defined by the equalities $\gb^{0}=0$ is
finite (and thus $S$ is compact), then the hierarchy of relaxation problems introduced by
Lasserre in \cite{Las01} is exact.  It is also proved that there is no
duality gap if the generators of this ideal satisfy some regularity conditions.
In \cite{NieReal}, it is proved that if the real variety defined by
the equalities $\gb^{0}=0$ is finite, then the hierarchy of relaxation
problems introduced by Lasserre is exact, this answers an open
question in \cite{LauSur}.

In a second family of methods, equality constraints which are
naturally satisfied by the minimizer points are added. These constraints are
for instance the gradient of $f$ when $S=\R^{n}$ or the
Karush-Kuhn-Tucker (KKT) constraints, obtained by introducing Lagrange
multipliers.  In \cite{NDS}, it is proved that a relaxation hierarchy
using the gradient constraints is exact when the gradient ideal is
radical.  In \cite{Marshall}, it is shown that this gradient hierarchy
is exact, when the global minimizers satisfy the Boundary Hessian
condition. In \cite{DNP}, it is proved that a relaxation hierarchy
which involves the KKT constraints is exact when the
KKT ideal is radical. In \cite{Ha-Pham:10}, a relaxation hierarchy
obtained by projection of the KKT constraints is proved to be exact
under a regularity condition on the {\em real} minimizer
points\footnote{The results of this paper are true but a problem
  appears in the proof which we fix in the present paper.}. In
\cite{Nie11}, a similar relaxation hierarchy is shown to be exact
under a stronger regularity condition for the {\em complex} points of
associated KKT varieties.
These regularity conditions require that the gradient of the active
constraints evaluated at the points of $S$ or of some complex varieties are linearly
independent. 
Thus they cannot be used for general semi algebraic sets $S$, for
instance when $S$ is a real non-complete intersection variety.

Moreover, the assumption that the minimum is reached at a KKT point is
required. Unfortunately, in some cases the set of KKT points of $S$
can be empty. As we shall see, this obstacle can be removed
using Fritz John variety (see \cite{FJohn48,MangFrom67}). 
There is not much work dedicated to this issue (see \cite{Lasserre:book}).

The case where the infimum value is not reached has also been studied. 
In \cite{Schweighofer06}, relaxation techniques are studied for
functions for which the minimum is not reached and which satisfy
some special properties ``at infinity''. 
In \cite{Ha-Pham:08}, tangency constraints are used in a relaxation hierarchy
which converges to the global minimum of a polynomial, when the
polynomial is bounded by below over $\R^{n}$.
In \cite{Guo:2010:GOP:1837934.1837960}, generic changes of
coordinates and a partial gradient ideal are used in a relaxation
hierarchy which also converges to the global minimum of $f$ on $\R^{n}$. 

In the cases studied so far, the exactness of the relaxation is proved
under a genericity condition or a compactness property.  From an
algorithmic point of view, the flat extension condition of
Curto-Fialkow \cite{CF96} is used in most of the works \cite{HeLa05,
  Lau07, LLR08b, LauSur} to detect the exactness of the hierarchy, when the
number of minimizers is finite.  In \cite{Nie2012}, it is proved that
the Curto-Fialkow flat extension criterion is eventually satisfied on
truncated moment matrices under some regularity conditions or
archimedean conditions.  In \cite{lasserre:hal-00651759}, a sparse
extension \cite{MoLa2008} of this flat extension condition is used to
compute zero-dimensional real radical ideals.

The second issue is related to the problem of computing all the
minimizer points, which is also important from a practical point of
view.  In \cite{LLR08b}, the kernel of moment matrices is used to
compute generators of the real radical of an ideal. This method is
improved in \cite{lasserre:hal-00651759} to compute a border basis of
the real radical, involving SDP problems of significantly smaller
size, when the real radical ideal is zero-dimensional. The case of
positive dimensional real radical ideal is analysed in \cite{Ros2009}
and \cite{Zhi}. The problem of computing the minimizer ideal for
general optimization problems from exact relaxation hierarchies has
not been addressed, though it is mentioned in \cite{Nie2012} for
zero-dimensional minimizer ideals

%
 
Notice that Problem \eqref{problem1} can be attacked from a 
purely algebraic point of view. It reduces to the
computation of a (minimal) critical value and polynomial
system solvers can be used to tackle it (see
e.g. \cite{Parrilo03minimizingpolynomial}, 
\cite{Greuet:2011:DRI:1993886.1993910}).
But in this case, the complex solutions of the underlying algebraic
system come into play and additional computation efforts should be
spent to remove these extraneous solutions.
Semi-algebraic techniques such as Cylindrical Algebraic
Decomposition or extensions \cite{ElDin:2008:CGO:1390768.1390781}
may also be considered here, providing algorithms to solve Problem
\eqref{problem1}, but suffering from similar issues.

\noindent{}\textbf{Contributions.}
Our aim is to show that for the general polynomial optimization problem
\eqref{problem1}, exact SDP relaxations can be constructed, which either 
detect that the problem is infeasible or compute the minimal value and
the ideal associated to the minimizer points.
The main contributions are the following:
\begin{itemize}
 \item We prove that exact relaxation hierarchies depending on the variables
 $\xx$ can be constructed for solving the optimization problem \eqref{problem1} (see
 Theorem \ref{thm:exact} and Theorem
 \ref{cthprin}). 
 \item We show that even if the minimizer points are not KKT points, we can find them 
 using the Fritz John variety (see Section 3.3 and Section 3.4). We
 describe an approach, which splits this minimizer set into
 the KKT minimizer set and the singular minimizer set which can be
 recursively computed using the same method.
\item We prove that if the set of KKT minimizers is empty, the
  SDP relaxation will eventually be empty (Theorem \ref{thm:exact}). 
\item We prove that the KKT minimizer ideal can be constructed from the
   moment matrix of an optimal linear form, when the corresponding
   relaxation is exact, even if the ideal is not zero-dimensional (Theorem \ref{cthprin}).
 \item We prove that the exactness of the relaxation 
 depends only on the real points which satisfy these constraints (Theorem \ref{cthprin}).
 \item We provide a general approach which allows us to treat in a uniform way
   and to extend results on the representation of
   polynomials which are positive (resp. non-negative) on the
   critical points (see \cite{DNP}  and Theorem 4.9) and on the exactness of
   relaxation hierarchies 
(see \cite{NDS}, \cite{Ha-Pham:08}, \cite{LLR08b}, \cite{Nie11},
\cite{lasserre:hal-00651759}, \cite{NieReal}  and Theorem 6.2, Theorem 6.4, Theorem 6.5, Theorem 6.6).
\end{itemize}
%
%

\noindent{}\textbf{Content.}
The paper is organized as follows. In Section 2, we recall
algebraic concepts and describe the hierarchy
of finite dimensional convex optimization problems considered. 
In Section 3, we analyse the varieties associated to the critical
points of the minimization problem. 
Section 4, is devoted to the representation of positive and non-negative
polynomials on the critical points as sum of squares modulo the gradient ideal.  
In Section 5, we prove that when the order of relaxation is big
enough, the sequence of finite dimensional convex optimization problems
attains its limit and the minimizer ideal can be generated from the
solution of our relaxation problem. 
In Section 6, we analyse some consequences of these results. 
Finally, Section 7 contains several examples which illustrate the approach.

\section{Ideals, varieties, optimization and relaxation}\label{sec:2}
In this section, we recall some algebraic concepts as ideals and
varieties and we set our notation.

\subsection{Ideals and varieties}
Let $\K[\xx]$ be the set of the polynomials in the variables
$\xx=(x_1,\ldots$, $x_n)$, with coefficients in the field
$\K$. Hereafter, we choose\footnote{For notational simplicity, we consider only these two fields, but $\R$ and $\C$ can
be replaced respectively by any real closed field and any field containing its algebraic closure.}
 $\K=\R$ or $\C$.
 Let $\overline{\K}$ denotes the algebraic closure of ${\K}$.
For $\alpha \in \N^n$, $\xx^{\alpha}= x_1^{\alpha_1} \cdots x_n^{\alpha_n}$ is the monomial with exponent $\alpha$
and degree $|\alpha|=\sum_i\alpha_i$.  The set of all monomials in $\xx$ is
denoted $\Mon = \Mon(\xx)$. 

For $t\in \N\cup \{\infty\}$ and $B\subseteq \K[\xx]$, we introduce the following sets:
\begin{itemize}
 \item $B_{t}$ is the set of elements of $B$ of degree $\le t$,
\item $\Span{B} = \big\{ \sum_{f\in B} \lambda_{f}\, f\ |\ f\in B, \lambda_f\in \K\big\}$ is the linear span of $B$,
\item $(B) = \big\{ \sum_{f\in B} p_f\, f \ | \ p_f \in \K[\xx], f \in B \big\}$ is the ideal in $\K[\xx]$ generated by $B$,
\item $\SpanDeg{B}{t} = \big\{ \sum_{f\in B_t} p_f\, f\ | \ p_f \in \K[\xx]_{t-\deg(f)}\big\}$ is the vector space spanned by $\{\xx^\alpha  f\mid
  f\in B_t, |\alpha|\le t-\deg(f)\}$, 
 \item $\Qc^{+}_{t}= \big\{ \sum_{i=1}^{l} p_{i}^{2}\mid l\in \N, p_{i}\in \R[\xx]_{t} \big\}$ is the set of finite sums of
   squares of polynomials of degree $\le t$; $\Qc^{+}= \Qc^{+}_{\infty}$.
\end{itemize}
By definition $\SpanDeg{B}{t} \subseteq (B)\cap \K[\xx]_t= (B)_t$, but the inclusion may be strict.

By convention, a set of constrains $C=\{c_1^{0},
\ldots, c_{n_{1}}^{0}$; $c_1^{+}, \ldots$, $c_{n_{2}}^{+}\} \subset \RR[\xx]$ is a
finite set of polynomials 
composed of a subset $C^{0}=\{c_1^{0},
\ldots, c_{n_{1}}^{0}\}$ corresponding to the equality constraints and 
a subset $C^{+}=\{c_1^{+}, \ldots, c_{n_{2}}^{+}\}$ corresponding to
the non-negativity constraints.
For two set of constraints $C,C'\subset \R[\xx]$, we say that $C \subset C'$ if $C^{0}\subset C'^{0}$ and $C^{+}\subset C'^{+}$.

\begin{defn} 
For $t\in \N\cup \{\infty\}$ and a set of constraints $C=\{c_1^{0},
\ldots, c_{n_{1}}^{0}$; $c_1^{+}, \ldots$, $c_{n_{2}}^{+}\} \subset \RR[\xx]$,  we define the (truncated) quadratic module
  of $C$ by
 \begin{equation*}
  \Qc_{t}(C)=\{\sum_{i=1}^{n_{2}} c_{i}^{0} \, h_{i}+s_{0}+\sum_{j=1}^{n_{2}}c^{+}_{j} \, s_{j} \mid
h_{i} \in \R[\xx]_{2t-\deg (c_{i}^{0})}, s_{0}\in \Qc^{+}_{t}, s_{i}
\in \Qc^{+}_{t- \lceil \deg (c^{+}_{i})/2  \rceil} \}.
 \end{equation*}
If $\tilde{C}$ is such that $\tilde{C}^{0}=C^{0}$ and
$\tilde{C}^{+}=\{\prod (c_{1}^{+})^{\epsilon_{1}}\cdots (c^{+}_{n_{2}})^{\epsilon_{n_{2}}} \mid 
\epsilon_{i} \in \{0,1\}\}$,  $\Qc_{t}(\tilde{C})$ is also called the
(truncated) preordering of $C$ and denoted $\Pc_{t}(C)$.  When $t=\infty$, $\Pc(C):=
\Pc_{\infty} (C)$ is the preordering of $C$.
The (truncated) preordering generated by the positive constraints is
denoted $\Pc^{+} (C)=\Pc (C^{+})$.
\end{defn}

\begin{defn}
 Given $t \in \NN\cup\{\infty\}$ and a set of constraints $C
 \subset \R[\xx]$, we define 
$$
  \Nc_{t}(C):=\{\Lambda \in (\R[\xx]_{2t})^{*} \mid \Lambda(p) \geq
  0,\ \forall p \in \Qc_{t}(C), \Lambda (1)=1 \}.
$$
When we replace $\Qc_{t} (C)$ by $\Pc_{t} (C)$ in this definition,
we denote the corresponding set by $\Lc_{t} (C)$.
\end{defn}



 

\medskip
Given a set $I\subseteq \K[\xx]$ and a field $\Le \supseteq \K$, we denote by
\begin{equation*}
 \Vc^{\Le}(I):=\{x\in \Le^n\mid f(x)=0\ \forall f\in I\}
\end{equation*}
its associated variety in $\Le^{n}$. By convention
$\Vc(I)=\Vc^{\clK}(I)$, where $\clK$ is the algebraic closure of $\K$. 
We  also consider sets of homogeneous equations $I$ and
the varieties $\PP\Vc (I)$ (resp. $\PP\Vc^{\R}(I)$) defined in the projective space $\PP^{n}$
 (resp.  the real projective space $\R\PP^{n}$).

For a set $V\subseteq \K^n$, we define its vanishing ideal
\[\Ic(V):=\{p\in \K[\xx]\mid p(v)=0\ \forall v\in V\}.\]
For a set $V\subset \Le^{n}$ with $\Le \supseteq \K$, $V^{\K}=V\cap \K^{n}$.
Hereafter, we  take $\K=\R$ and $\Le=\C$, so that
$\Vc(I)=\Vc^{\C}(I)$, $\Vc^\R(I)=\Vc (I)^{\R}= \Vc(I)\cap \R^n.$

\begin{defn}
For a set of constrains $C= (C^{0};C^{+})\subset \R[\xx]$,
\begin{eqnarray*}
\Sc (C) &:=& \{\xx \in \R^{n} \mid c^0(\xx)=0 \ \forall c^0 \in C^0,\ c^{+}(\xx)\geq 0\ \forall c^{+}\in C^{+}\},\\
\Sc^{+}(C) &:= &\{\xx \in \R^{n} \mid  c^{+}(\xx)\geq 0\ \forall c^{+}\in C\}.
\end{eqnarray*}
\end{defn}
To describe the vanishing ideal of these sets, we introduce the
following ideals:
\begin{defn} For a set of constraints $C= (C^{0};C^{+}) \subset \R[\xx]$,
\begin{eqnarray*}
\sqrt{C^{0}} &=& \{p\in \R[\xx] \,\mid\, p^m\in (C^{0})\ \text{ for
  some } m\in \N\setminus \{0\}\}\\
\sqrt[\R]{C^{0}}&=&\{p\in \R[\xx]\mid p^{2m} + q \in (C^{0}) \
\text{ for some } m\in \N\setminus \{0\}, q \in \Qc^{+} \}\\
\sqrt[C^{+}]{C^{0}}&=&\{p\in \R[\xx]\mid p^{2m} + q \in
(C^{0}) \ \text{ for some } m\in \N\setminus \{0\},  q \in \Pc^{+}(C) \}
\end{eqnarray*}
These ideals are called respectively the radical of $C^{0}$, the real
radical of $C^{0}$, the $C^{+}$-radical of $C^{0}$.
\end{defn}

\begin{remark}
 If $C^{+}=\emptyset$, then $\sqrt[C^{+}]{C^{0}}=\sqrt[\R]{C^{0}}$.
\end{remark}

The following three famous theorems relate vanishing and radical
ideals:
\begin{thm}\label{null}Let $C= (C^{0};C^{+})$ be a set of
  constraints of $\R[\xx]$. 
\begin{itemize}
\item[(i)] {\bf Hilbert's Nullstellensatz} (see, e.g., \cite[\S 4.1]{CLO97})
  $\sqrt{C^{0}}=\Ic(\Vc^{\C}(C^{0}))$.
\item[(ii)] {\bf Real Nullstellensatz} (see, e.g., \cite[\S 4.1]{BCR98})
$ \sqrt[\R]{C^{0}}=\Ic(\Vc^\R({C^{0}}))$.
\item[(iii)]{\bf Positivstellensatz} (see, e.g., \cite[\S 4.4]{BCR98})
$\sqrt[C^{+}]{C^{0}}=\Ic (\Sc (C)) = \Ic(\Vc^\R({C^{0}}) \cap \Sc^{+}(C) )$.
\end{itemize}
\end{thm}

\subsection{Relaxation hierarchy}
The approach proposed by Lasserre in \cite{Las01} to solve Problem
\eqref{problem1} consists in
approximating the optimization problem by a sequence of finite
dimensional convex optimization problems, 
which can be solved efficiently by Semi-Definite Programming tools.
This sequence is called Lasserre hierarchy of relaxation problems.
Let $C$ be a set of constraints in $\R[\xx]$ such that $S(C)=S(\gb)$.   
Hereafter, we consider the relaxation hierarchy
associated to preordering sequences: 
$$ 
\cdots \subset \Lc_{t+1} (C) \subset
\Lc_{t} (C) \subset \cdots 
\ \mathrm{and}\  
\cdots \subset \Pc_{t} (C)  \subset
\Pc_{t+1} (C) \subset \cdots 
$$
These convex sets are used to define extrema that approximate the
solution of the minimization problem \eqref{problem1}.
\begin{defn}\label{mintruncated} 
 Let $t \in \NN$ and let $C$ be the set of constraints in
 $\R[\xx]$. We define the following extrema: 
  \begin{itemize}
  \item $f^*_{C}= \inf_{\xx \in \Sc (C)} f(\xx),$
  \item $f^{\mu}_{t,C}= \inf \ \{ \Lambda(f)$ s.t.  $\Lambda
    \in\Lc_{t} (C) \},$
  \item $f^{sos}_{t,C}= \sup \ \{ \gamma \in \R$ s.t. $f-\gamma \in  \Pc_{t}(C)\}.$  
 \end{itemize}
 By convention if the corresponding sets are empty, 
$f_{C}^{*}=-\infty$, 
$f^{sos}_{t,C}=-\infty$ and $f^{\mu}_{t,C}=+\infty$.
\end{defn}

\begin{remark}\label{rem:fmin}
We have $f^{sos}_{t,C} \le f^{\mu}_{t,C}\le f^{*}_{C}$. 

Indeed, if there exists $\gamma \in \R$ such that $f- \gamma = q \in \Pc_{t}(C)$ 
 then $\forall \Lambda \in \Lc_{t}(C)$, $\Lambda(f-\gamma)=
 \Lambda(f)- \gamma =  \Lambda(q) \ge 0$, which proves the first inequality. 
 
Since for any $\mathbf{s} \in S$, the
 evaluation ${\mathbf{1}}_{\mathbf{s}}: p\in \R[\xx]\mapsto p (\mathbf{s})$ is 
 in $\Lc_{t} (C)$, we have ${\mathbf{1}}_{\mathbf{s}} (f)= f (\mathbf{s})\ge
 f^{\mu}_{t,C}$. This proves the second inequality.
\end{remark}
 
As $\Lc_{t+1}(C) \subset \Lc_{t}(C)$ and $\Pc_{t}(C) \subset \Pc_{t+1}(C)$
we have the following increasing
sequences for $t \in \N$:
$$
 \cdots f^{\mu}_{t,C} \le f^{\mu}_{t+1,C} \le \cdots \le f_{C}^*
 \ \mathrm{and}\   
\cdots f^{sos}_{t,C} \le f^{sos}_{t+1,C} \le \cdots \le f_{C}^*.
$$
The foundation of Lasserre relaxation method is to show that these sequences
converge to $f^{*}_{C}$, see \cite{Las01}. 








We are interested in constructing hierarchies for
which, the minimum $f^{*}_{C}$ is reached in a finite number of steps.
Such hierarchies are called {\em exact}. 
We are also interested to compute the minimizers points. For that
purpose, we introduce now the truncated Hankel operators, which  play a central
role in the construction of the minimizer ideal of $f$ on $S$.
\begin{defn}
 For  $t \in \NN$ and a linear form $\Lambda \in (\R[\xx]_{2t})^{*}$, we define the truncated Hankel operator
 as the map $M_{\Lambda}^t : \R[\xx]_{t} \rightarrow
 (\R[\xx]_{t})^{\ast}$ such that $M_{\Lambda}^t(p)(q) = \Lambda(p\,q)$ for $p,q\in \R[\xx]_{t}$.
 Its matrix in monomial bases of $\R[\xx]_{t}$ and $(\R[\xx]_{t})^{*}$ is also called the moment
 matrix of $\Lambda$.
\end{defn}

The kernel of the truncated Hankel operator is
\begin{equation}
  \ker M_{\Lambda}^t=\{p \in \R[\xx]_{t} \mid \ \Lambda(p\,q)=0 \ \forall q \in \R[\xx]_{t} \}.
\end{equation}


Given $t \in \NN$ and $C=\{0\}$ and $\Lambda,\Lambda' \in
\R[\xx]_{2t}^*$,  we easily check the following properties:
\begin{itemize}
 \item $\forall p\in \R[\xx]_{t}$, $\Lambda(p^2) = 0$ implies $p \in \ker M_{\Lambda}^{t}$.
 \item $\ker M_{\Lambda+\Lambda'}^{t}=\ker M_\Lambda^{t} \cap \ker M_{\Lambda'}^{t}$. 
\end{itemize}




The kernel of truncated Hankel operator 
is used to compute generators of the minimizer ideal, as we will see.

\section{Varieties of critical points}
Before describing how to compute the minimizer points, we analyse the
geometry of this minimization problem and the varieties associated
to its critical points.  
In the following, we  denote by $\yy=(\xx,\uu,\vv)$ and
$\zz=(\xx,\uu,\vv,\s)$, the $n+n_{1}+n_{2}$ and $n+n_{1}+2 n_{2}$
variables of these problems. For any ideal $J\subset \RR[\zz]$, we 
denote $J^{\xx} = J \cap \RR[\xx]$.
The projection of $\C^{n}\times \C^{n_{1}+2\,n_{2}}$ (resp. $\C^{n}\times \C^{n_{1}+\,n_{2}}$) on $\C^{n}$
is denoted  $\pi^{\xx}$.

\subsection{The gradient variety}
A natural approach to deal with constraints in optimization problems is to introduce
Lagrangian multipliers.
Replacing the inequalities $g_{i}^{+} \ge 0$ by the equalities
$g_{i}^{+}-s_{i}^2=0$ (adding new variables $s_{i}$) and introducing
new parameters for all the equality constraints yields the following minimization problem:
\begin{eqnarray}\label{problem3}
\inf_{(\xx,\uu,\vv,\s)\in \R^{n}\times \R^{n_{1}+ 2\, n_{2}}}&& f(\xx) \\
s.t. && \nabla F(\xx,\uu,\vv,\s) = 0    \nonumber 
\end{eqnarray}
where $F(\xx,\uu,\vv,\s) = f(\xx)-\sum_{i=1}^{n_1} u_i
g_i(\xx)-\sum_{j=1}^{n_2} v_j  (g_{j}^{+}(\xx)-s_{j}^2)$, $\uu=(u_1,...,u_{n_1}), \ \vv=(v_1,...,v_{n_2})$ and  $\s=(s_1,...,s_{n_2})$.\\

\begin{defn}
The gradient ideal of $F(\zz)$ is:
\begin{equation*}
 I_{grad}=(\nabla F
 (\zz))=(F_1,...,F_n,g_{1}^{0},...,g_{n_1}^{0},g_{1}^{+}-s_{1}^2,...,g_{n_2}^{+}-s_{n_2}^2,v_1
 s_1,...,v_{n_2} s_{n_2}) \subset \R[\zz]
\end{equation*}
where $F_i=\frac{\partial f}{\partial x_i}-\sum_{j=1}^{n_1} u_j
\frac{\partial g_{j}^{0}}{\partial x_i}-\sum_{j=1}^{n_2} v_j \frac{\partial
  g_{j}^{+}}{\partial x_i}$.
The gradient variety is $V_{{grad}} = \Vc (I_{grad})$ and
we denote $V^{\xx}_{grad} = \overline{\pi^{\xx} (V_{grad})}$.
\end{defn}

\begin{defn}
For any $F\in \R[\zz]$, the values of $F$ at the (resp. real) points of
$\Vc (\nabla F)= V_{grad}$ are called the
(resp. real) critical values of $F$. 
\end{defn}

We easily check the following property:
\begin{lem}\label{fegalF}
$F\mid_{V_{{grad}}}=f\mid_{V_{{grad}}}$.
\end{lem}
Thus minimizing $F$ on $V_{grad}$ is the same as minimizing $f$ on
$V_{grad}$, that is computing the minimal critical value of $F$.

\subsection{The Karush-Kuhn-Tucker variety}
In the case of a constrained problem, one usually introduce the
Karush-Kuhn-Tucker (KKT) constraints: 
\begin{defn}
A point $\xx^{*}$ is called a KKT point if there exists
$u_{1}, \ldots, u_{n_{1}}, v_{1}, \ldots, v_{n_{2}} \in \RR$ s.t. 
$$ 
\nabla f (\xx^{*}) - \sum_{i=1}^{n_{1}}
u_{i} \nabla g_{i}^{0}(\xx^{*}) - \sum_{j=0}^{n_{2}} v_{j} \nabla g_{j}^{+}
(\xx^{*})=0, \  g_{i}^{0}(\xx^{*})=0, \ v_{j} g_{j}^{+}
(\xx^{*})=0.
$$
\end{defn}
 
The corresponding minimization problem is the following:
\begin{eqnarray}\label{problem4}
\inf_{(\xx,\uu,\vv)\in \R^{n+n_{1} n_{2}}}&& f(\xx) \\
s.t. &&  F_1=\cdots=F_n=0 \nonumber \\
     && g_{1}^{0}=\cdots=g_{n_1}^{0}=0 \nonumber \\
     && v_1\,g_{1}^{+}=\cdots=v_{n_2}\,g_{n_2}^{+}=0 \nonumber \\
     && g_{1}^{+}\ge 0,\ldots, g_{n_2}^{+} \ge 0 \nonumber 
\end{eqnarray}
where $F_i=\frac{\partial f}{\partial x_i}-\sum_{j=1}^{n_1} u_j \frac{\partial g_{j}^{0}}{\partial x_i}-\sum_{j=1}^{n_2} v_j \frac{\partial g_{j}^{+}}{\partial x_i}$.\\

This leads to the following definitions:
\begin{defn}
The Karush-Kuhn-Tucker (KKT) ideal associated to Problem \eqref{problem1} is
\begin{equation}
I_{KKT}=(F_1,...,F_n,g_{1}^{0},...,g_{n_1}^{0},v_1 g_{1}^{+},...,v_{n_2}g_{n_2}^{+})\subset \R[\yy].
\end{equation}
The KKT variety is $V_{KKT} = \Vc (I_{KKT}) \subset
\C^{n}\times \C^{n_{1}+n_{2}}$ and the real KKT variety is $V_{KKT}^{\R} =
V_{KKT} \cap (\R^{n}\times \R^{n_{1}+n_{2}})$. Its projection on $\xx$ is
$V^{\xx}_{KKT} = \overline{\pi^{\xx} (V_{KKT})}$, where $\pi^{\xx}$ is the
projection of $\C^{n}\times \C^{n_{1}+n_{2}}$ onto $\C^{n}$. 
\end{defn}
The set of KKT points of $S$ is denoted $S_{KKT}$ and a KKT-minimizer
of $f$ on $S$ is a point $\xx^{*} \in S_{KKT}$ such that $f (\xx^{*}) = \min_{\xx
 \in S_{KKT}} f (\xx)$.

Notice that $V^{\xx,\R}_{KKT}=\overline{\pi^{\xx}
  (V_{KKT})}^{\R}=\overline{\pi^{\xx} (V_{KKT}^{\R})}$, since any
linear dependency relation between real vectors can be realized with real coefficients.
 
The KKT ideal is related to the gradient ideal as follows:
\begin{prop}\label{prop:proj}
 $I_{KKT}=I_{grad} \cap \R[\yy]$.
\end{prop}
\begin{proof}
As  $s_i(s_i v_i)  + v_i (g_{i}^{+}-s_{i}^{2})= v_i g_{i}^{+}\ \forall
i=1,...,n_2$, we have $I_{KKT} \subset I_{grad} \cap
\R[\yy]$.

In order to prove the equality, we use the property that if $K$ is a Groebner basis of $I_{grad}$
for an elimination ordering such that $\s \gg \xx, \uu, \vv$ then
$K \cap \R[\yy]$ is the Groebner basis of  $I_{grad} \cap
\R[\yy]$ (see \cite{CLO97}).
Notice that $s_i(s_i v_i) + v_i (g_{i}^{+}-s_{i}^{2})= v_i g_{i}^{+}$ ($i=1,...,n_2$) are the only S-polynomials
involving the variables $s_{1}, \ldots, s_{n_{2}}$ which may have a
non-trivial reduction. Thus $K \cap
\R[\yy]$ is also the Groebner basis of
$F_1,...,F_n,g_{1}^{0},...,g_{n_1}^{0},v_1 g_{1}^{+},...,v_{n_2}g_{n_2}^{+}$ and we have
$(K)  \cap \R[\yy]= I_{grad} \cap \R[\yy] = I_{KKT}$.
\end{proof}

The KKT points on $S$ are related to the real points of the gradient variety as follows:
\begin{lem}\label{rem:proj} $S_{KKT} = \pi^{\xx} (V_{grad}^{\R})=V_{grad}^{\xx, \R}\cap \Sc^{+} (\gb)$.
\end{lem}
\begin{proof}
A real point $\yy= (\xx,\uu,\vv)$ of $V_{KKT}^{\R}$ lifts to a
 point $\zz= (\xx,\uu,\vv, \s)$ in $V_{grad}^{\RR}$, if and only if, $g_{i}^{+} (\xx)\geq 0$ for $i=1, \ldots,
 n_{2}$. This implies that $V_{KKT}^{\RR} = \pi^{\yy}
 (V_{grad}^{\R})\cap \Sc^{+} (\gb)$, which gives by projection the
 equalities $S_{KKT}= \pi^{\xx} (V_{grad}^{\R})\cap \Sc^{+} (\gb)=
 \pi^{\xx} (V_{grad}^{\R})$ since a point $\xx$ of $V_{grad}^{\RR}$
 satisfies $\gb_{j}^{+} (\xx)\ge 0$ for $j\in [1,n_{2}]$.
\end{proof}

This shows that if a minimizer point of $f$ on $S$ is a KKT point,
then it is the projection of a real critical point of $F$.

\subsection{The Fritz John variety}

A minimizer of $f$  on $S$ is not necessarily a KKT point. 
More general conditions that are satisfied by minimizers were given by F. John for polynomial
non-negativity constraints and further refined for general polynomial constraints
\cite{FJohn48,MangFrom67}.
To describe these conditions, we introduce a new variable $u_{0}$ and
denote by $\yy'$ the set of variables $\yy'=(\xx,u_{0},\uu,\vv)$. Let 
$F_i^{u_{0}}=u_{0}\, \frac{\partial f}{\partial
  x_i}-\sum_{j=1}^{n_1} u_j \frac{\partial g_{j}^{0}}{\partial
  x_i}-\sum_{j=1}^{n_2} v_j \frac{\partial g_{j}^{+}}{\partial x_i}$.
\begin{defn} For any $\gamma \subset [1,n_1]$, let
\begin{equation}
I_{FJ}^{\gamma}=(F_1^{u_{0}},...,F_n^{u_{0}},g_{1}^{0},...,g_{n_1}^{0},v_1
g_{1}^{+},...,v_{n_2}g_{n_2}^{+}, u_{i}, i\not\in \gamma)\subset \R[\yy'].
\end{equation}
For $m\in \NN$, the $m^{\mathrm{th}}$ Fritz-John (FJ) ideal associated to Problem \eqref{problem1} is
\begin{equation}
I_{FJ}^{m}=\cap_{|\gamma|= m} I_{FJ}^{\gamma}.
\end{equation}
Let  $V_{FJ}^{\gamma}=\Vc(I_{FJ}^{\gamma}) \subset \C^{n}\times \PP^{n_{1}+n_{2}}$.
The $m^{\mathrm{th}}$ FJ variety is $V_{FJ}^{m} = \Vc (I_{FJ}^{m})=
\cup_{|\gamma|= m} V_{FJ}^{\gamma}$,
and the real FJ variety is $V_{FJ}^{m,\R} =V_{FJ}^{m} \cap \R^{n}\times \R\PP^{n_{1}+n_{2}}$. Its projection on $\xx$ is
$V^{m,\xx}_{FJ} = \pi^{\xx} (V_{FJ}^{m}) = \overline{\pi^{\xx}
  (V_{FJ}^{m})}$.
When $m=\max_{\xx \in S} \ \rank([ \nabla g_{1}^0(\xx),\ldots,\nabla g_{n_1}^0(\xx)])$, the $m^{\mathrm{th}}$ FJ variety is denoted $V_{FJ}$.
\end{defn}
Notice that this definition slightly differs from the classical one
\cite{FJohn48,Lasserre:book,MangFrom67}, which does not provide any
information when the gradient vectors $\nabla g_{i}^0(\xx), i=1\ldots n_{1}$ are
linearly dependent on $S$.
\begin{prop}\label{prop:FJ}
Any minimizer $\xx^{*}$ of $f$ on $S$ is the projection of a real point of $V_{FJ}^{\R}$.
\end{prop}
\begin{proof}
 The proof is similar to Theorem 4.3.2 of \cite{MangFrom67}. At a
 minimizer point $\xx^{*}$ (if it exists) We
 consider a maximal set of linearly independent gradients $\nabla
 g_{j}^0(\xx^*)$ for $j \in \gamma$ 
 (with $|\gamma|\leq m$) and apply the same proof as \cite{MangFrom67}[Theorem 4.3.2].
This shows that $\xx^{*} \in V_{FJ}^{\gamma, \RR} \subset V_{FJ}^{\RR}$.
\end{proof}

\begin{defn} We denote by $V_{sing}=V_{FJ}\cap \Vc (u_{0})$ the intersection of $V_{FJ}$ with the hyperplane $u_{0}=0$.
\end{defn}
We easily check that the ``affine part'' of $V_{FJ}$ corresponding to
$u_{0}\neq 0$ is the variety $V_{KKT}$. 
Thus, we have the decomposition
$$ 
V_{FJ} = V_{sing} \cup V_{KKT},
$$ 
Its projection on $\C^{n}$ decomposes as 
\begin{equation}\label{dec:FJ}
V_{FJ}^{\xx} = V_{sing}^{\xx} \cup V_{KKT}^{\xx}.
\end{equation}

Let us describe more precisely the projection $V^{\xx}_{FJ}$ onto $\C^{n}$.
For $\nu =\{j_{1}, \ldots, j_{k}\}\subset  [1, n_{2}] $, we define 
\begin{eqnarray*}
A_{\nu} &=& [ \nabla f, \nabla g_{1}^{0}, \ldots, \nabla g_{n_1}^{0}, \nabla
g_{j_{1}}^{+}, \ldots, \nabla g_{j_{k}}^{+}]\\
V_{\nu}&= &\{ \xx \in \C^{n}\mid g_{1}^{0} (\xx) =0, i=1\ldots n_{1},
\ g_{j}^{+} (\xx) =0, j \in \nu, \rank (A_{\nu}) \le m+ |\nu|\}.
\end{eqnarray*}
Let $\Delta_{1}^{\nu}, \ldots, \Delta_{l_{\nu}}^{\nu}$ be polynomials
defining the variety $\{ \xx \in \C^{n}\mid \rank (A_{\nu}) \le m+
|\nu|\}$. 
If $n> m+|\nu|$, these polynomials can be chosen as linear combinations of $(m+ |\nu|+1)$-minors of the matrix $A_{\nu}$,
as described in \cite{BrunsVetter88,Nie11}. If $n\le m+|\nu|$, we
take $l_{\nu}=0$, $\Delta_{i}^{\nu}=0$.
Let $\Gamma_{FJ}$ be the union of $\gb^{0}$ and the set of polynomials
\begin{equation}\label{eq:kktproj}
g_{\nu,i} := \Delta_{i}^{\nu} \prod_{j \not \in \nu} g_{j}^{+},
\end{equation}
for $i=1, \ldots, l_{\nu}, \nu \subset [0, n_{2}]$.
\begin{lem}\label{lem:kktproj} $V^{\xx}_{FJ}= \cup_{\nu \subset [0,
    n_{2}]} V_{\nu} = \Vc (\Gamma_{FJ})$.
\end{lem}
\begin{proof}
For any $\xx\in \C^{n}$, let $\nu (\xx)=\{ j\in [1, n_{2}] \mid
g_{j}^{+} (\xx) =0 \}$.

Let $\yy'$ be a point of $V_{FJ}$, $\xx$ its projection on $\C^{n}$
and $\nu (\xx) = \nu =\{j_{1}, \ldots, j_{k}\}$. 
We have $g_{j}^{+} (\xx) \neq 0$, $v_{j}=0$ for $j\not\in \nu$ and
 $\Delta_{i}^{\nu}=0$  for $i=1, \ldots, l_{\nu}$. This implies
that $\rank (A_{\nu} (\xx)) \le m+ |\nu|$ and there exists
$(u_{0},u_{1}, \ldots, u_{n_{1}}, v_{1}, \ldots, v_{n_{2}})\neq 0$ and
$\gamma\subset [1,n_{1}]$ of size $|\gamma|\le m$ such that 
$$ 
u_{0} \nabla f + u_{1} \nabla g_{1}^{0} + \cdots + u_{n_{1}} \nabla
g_{n_{1}}^{0}+ v_{1} \nabla g_{j_{1}}^{+}+ \cdots + v_{n_{2}} \nabla g_{j_{k}}^{+}=0,
$$
with $u_{i}=0$, $i \not\in \gamma\subset [1,n_{1}]$.
Therefore $\xx \in \pi^{\xx} (V_{FJ})$, which proves that $\Vc (\gb^{0},g_{\nu,i},
\nu \subset [0, n_{2}],  i=1 \ldots l_{\nu})\subset \pi^{\xx}
(V_{FJ})$.

Conversely, if $\xx \in \pi^{\xx} (V_{FJ})$ then $\xx\in V_{\nu
  (\xx)} \subset \cup_{\nu} V_{\nu}$ which is defined by the polynomials
$g_{1}^{0}, \ldots, g_{n_{1}}^{0}$ and 
$g_{\nu,i} := \Delta_{i}^{\nu} \prod_{j \not \in \nu} g_{j}^{+}$,
for $i=1, \ldots, l_{\nu}, \nu \subset [0, n_{2}]$.
\end{proof} 
\begin{remark}\label{rem:realkkt}
The real variety $\pi^{\xx}(V^{\R}_{FJ})=V^{\xx}_{FJ}\cap \R^{n}$ can
also be defined by $\gb^{0}$ and the set $\Phi_{FJ}$ of polynomials
\begin{equation}\label{eq:kktprojreal}
g_{\nu} := \Delta^{\nu} \prod_{j \not \in \nu} g_{j}^{+} \ \mathrm{where}\ \Delta^{\nu} = \det (A_{\nu} A_{\nu}^{T}),
\end{equation}
for $\nu \subset [0, n_{2}]$ and $n> m+|\nu| $, as described in \cite{Ha-Pham:10}.
\end{remark}
 
Similarly the projection $V^{\xx}_{sing}$ onto $\C^{n}$ can be
described as follows.
For $\nu =\{j_{1}, \ldots, j_{k}\}\subset  [1, n_{2}] $, 
\begin{eqnarray*}
B_{\nu} &=& [ \nabla g_{1}^{0}, \ldots, \nabla g_{n_1}^{0}, \nabla
g_{j_{1}}^{+}, \ldots, \nabla g_{j_{k}}^{+}]\\
W_{\nu}&= &\{ \xx \in \C^{n}\mid g_{1}^{0} (\xx) =0, i=1\ldots n_{1},
 \ g_{j}^{+} (\xx) =0, j \in \nu, \rank (B_{\nu}) \leq m+ |\nu|-1\}.
\end{eqnarray*}
Let $\Theta_{1}^{\nu}, \ldots, \Theta_{l_{\nu}}^{\nu}$ be polynomials
defining the variety $\{ \xx \in \C^{n}\mid \rank (B_{\nu}) \le m+
|\nu|-1\}$ and 
let $\Gamma_{sing}$ be the union of $\gb^{0}$ and the set of polynomials
\begin{equation}\label{eq:sing}
\sigma_{\nu,i} := \Theta_{i}^{\nu} \prod_{j \not \in \nu} g_{j}^{+},
\end{equation}
for $\nu \subset [0, n_{2}],  i=1 \ldots l_{\nu}$.

We similar arguments, we prove the following
\begin{lem} $V^{\xx}_{sing}= \cup_{\nu \subset [0, n_{2}]} W_{\nu} =  \Vc (\Gamma_{sing})$.
\end{lem}

\subsection{The minimizer variety}
By the decomposition \eqref{dec:FJ} and Proposition \ref{prop:FJ}, we know
that the minimizer points of $f$ on $S$ are in 
\begin{equation}\label{eq:decomp}
S_{FJ} = S_{KKT} \cup  S_{sing} 
\end{equation}
where 
$S_{FJ}= \pi^{\xx} (V_{FJ}^{\R})\cap S= \pi^{\xx}(V_{FJ}^{\R})\cap \Sc^{+} (\gb)$, 
$S_{KKT}= \pi^{\xx} (V_{KKT}^{\R})\cap S= \pi^{\xx}(V_{KKT}^{\R})\cap \Sc^{+} (\gb)$, 
$S_{sing}= \pi^{\xx} (V_{sing}^{\R})\cap S=
\pi^{\xx}(V_{sing}^{\R})\cap \Sc^{+} (\gb)$.
Therefore, we can decompose the initial optimization problem
\eqref{problem1} into two
subproblems:
\begin{enumerate}
 \item find the infimum of $f$ on $S_{KKT}$;
 \item find the infimum of $f$ on $S_{sing}$;
\end{enumerate}
and take the least of these two infima. Since the second problem is of
the same type as \eqref{problem1} but with the additional constraints
$\sigma_{\nu,i}=0$ described in \eqref{eq:sing}, we analyse only the first
subproblem. The approach developed for this first sub-problem is
applied recursively to the second subproblem, in order to obtain the
solution of Problem \eqref{problem1}.

\begin{defn} We define the KKT-minimizer set and ideal of $f$ on $S$ as:
\begin{eqnarray*}
 S_{min} &=& \{\xx^{*}\in S_{KKT}\ \mathrm{s.t.}\ \forall \xx \in S_{KKT}, f
 (\xx^{*}) \le f (\xx) \}\\
 I_{min} &=& \Ic (S_{min}) \subset \R[\xx].
\end{eqnarray*}
\end{defn}
A point $\xx^*$ in $S_{min}$ 
is called a KKT-minimizer. 
Notice that $I_{KKT}\subset I_{min}$ and that $I_{min}$ is a real radical ideal.

We have $I_{min} \neq (1)$, if and only if, the KKT-minimum $f^{*}$ is reached in $S_{KKT}$.

If $n_{1}= n_{2}=0$, $I_{min} $ is the vanishing ideal of the {\em
 critical points} $\xx^{*}$ of $f$ (satisfying $\nabla f (\xx^{*})=0$) where $f (\xx^{*})$ reaches its minimal critical value.

\begin{remark}\label{rem:f=0}
If we take $f=0$ in the minimization problem (\ref{problem1}),
then all the points of $S$ are KKT-minimizers and $I_{min}= \Ic
(S) = \sqrt[\gb^{+}]{\gb^{0}}$.
Moreover, $I_{KKT}\cap \R[\xx]= (g_{1}^{0}, \ldots, g_{n_{1}}^{0})=(\gb^{0})$ since $F_{1}, \ldots, F_{n},
v_{1} g_{1}^{+}, \ldots, v_{n_{2}} g_{n_{2}}^{+}$ are homogeneous of degree 1
in the variables $\uu, \vv$.
\end{remark}

\section{Representation of positive polynomials}
In this section, we analyse the decomposition of
polynomials as sum of squares modulo the gradient ideal. 
Hereafter, $J_{grad}$ is an ideal of $\R[\zz]$
such that $\Vc (J_{grad}) = V_{grad}$ and $C$ is a set
of constraints in $\R[\xx]$ such that $\Sc^{+} (C)=\Sc^{+} (\gb)$.

The first steps consists in decomposing $V_{{grad}}$ in
components on which $f$ has a constant value. We recall here a result,
which also appears (with slightly different
hypotheses) in \cite[Lemma 3.3]{NDS}\footnote{In its proof, the Mean
  Value Theorem is applied for a complex valued function, which is not
valid. We correct the problem in the proof of Lemma \ref{cte}.}.

\begin{lem}\label{cte}
Let $f\in \R[\xx]$ and let $V$ be an irreducible subvariety contained in $\Vc^{\C}(\nabla f)$. Then $f(x)$ is constant on $V$.
\end{lem}
\begin{proof}
 If $V$ is irreducible in the Zariski topology induced from
 $\C[\xx]$, then it is connected in the strong topology on $\C^n$
and even piecewise smoothly path-connected \cite{Shafa}.
 Let $x, \ y$ be two arbitrary points of $V$. There exists a piecewise
 smooth path $\varphi(t) \ (0\le t\le 1)$ lying inside $V$ such that
 $x=\varphi(0)$ and $y=\varphi(1)$. Without loss of generality, we can
 assume that $\varphi$ is smooth between $x$ and $y$ in order to prove
 that $f (x)= f (y)$.
 By the Mean Value Theorem, it holds that for some $t_1\in (0,1)$
 \begin{equation*}
  \re(f(y)-f(x))= \re(f(\varphi (t)))' (t_{1}) = \re((\nabla f(\varphi(t_1))* \varphi'(t_1)))=0
 \end{equation*}
 since  $\nabla f$ vanishes on $V$. Then $\re(f(y))=\re(f(x))$. We
 have the same result for the imaginary part: for some $t_2\in (0,1)$
 \begin{equation*}
  \im(f(y)) - \im (f(x))=\im(f(\varphi(t)))' (t_2)=\im((\nabla
  f(\varphi(t_2)) * \varphi'(t_2)))=0
 \end{equation*}
 since  $\nabla f$ vanishes on $V$. Then $\im(f(y))=\im(f(x))$. We conclude that $f(y)=f(x)$ and hence $f$ is constant on $V$.
\end{proof}

\begin{lem}\label{pis}
The ideal $J_{grad} $ can be decomposed as 
$J_{grad} = J_0 \cap J_1 \cap \cdots \cap J_s$ with $V_{i}=\Vc
(J_{i})$ and $W_{i}= \overline{\pi^{\xx} ( V_{i})}$ where $\pi^{\xx} ( V_{i})$ is the projection of $V_{i}$ 
on $\C^n$ such that 
\begin{itemize}
 \item $f(V_j)=f_j\in \C$, $f_i \neq f_j$ if $i\neq j$, 
 \item $W_{i}^{\R} \cap \Sc^{+}(C) \neq \emptyset$ for $i=0, \ldots, r$,
 \item $W_{i}^{\R} \cap \Sc^{+}(C) =\emptyset$ for $i=r+1, \ldots, s$,
 \item $f_{0}< \cdots < f_{r}$.
\end{itemize}
\end{lem}
\begin{proof}
Consider a minimal primary decomposition of $J_{grad }$:
$$
J_{grad}= Q_{0} \cap \cdots \cap Q_{s'},
$$ 
where $Q_{i}$ is a primary component, and $\Vc (Q_{i})$ is an
irreducible variety in $\C^{n+n_1+2n_2}$ included in $V_{grad} $.
By Lemma \ref{cte}, $F$ is constant on $\Vc (Q_{i})$.
By Lemma \ref{fegalF}, it coincides with $f$ on each variety $\Vc (Q_{i})$.
We group the primary components $Q_{i}$ according to the values
$f_{0}, \ldots, f_{s}$ of $f$ on these components, into $J_{0},
\ldots, J_{s}$ so that $f (\Vc (J_{j}))=f_{i}$ with $f_{i}\neq f_{j}$ if
$i\neq j$.

We can number them so that $\overline{\pi^{\xx} (V_i)}^{\R}\cap \Sc^{+}(C)$ is
empty for $i=r+1, \ldots, s$ and contains a real point $\xx_{i}$ for $i=0,\ldots,
r$. Notice that such a point $\xx_{i}$ is in $\Sc$, since it satisfies
$g^{0} (\xx_{i})=0 \ \forall g^{0} \in C^{0}$ and $g^{+} (\xx_{i})\geq 0 \ \forall g^{+} \in C^{+}$.
As it is the limit of the projection of points in $\Vc (J_{i})$ on which $f$ is constant, we
have $f_{i}= f (\xx_{i})\in \R$ for $i=0, \ldots, r$.
We can then order $J_{0}, \ldots ,J_{r}$ so that $f_{0}<\cdots< f_{r}$.
\end{proof}

\begin{remark} If the minimum of $f$ on $S$  is reached at a
  KKT-point, then we have $f_{0}=\min_{\xx \in S} f (\xx)$.
\end{remark}

\begin{remark}\label{rpis}
If $V_{grad}^{\R}=\emptyset$, then for all $i=0, \ldots, s$, $W_{i}^{\R} \cap \Sc^{+}(C) =\emptyset$
and by convention, we take $r=-1$.
\end{remark}

\begin{lem}\label{lem:lagrange}
There exist $p_{0}, \ldots, p_{s} \in \C[\xx]$ such that
\begin{itemize}
 \item $\sum_{i=0}^{s} p_{i} = 1 \mod J_{grad} $,
 \item $p_{i} \in \bigcap_{j\neq i} J_{j}$,
 \item $p_{i}\in \R[\xx]$ for $i=0, \ldots, r$.
\end{itemize}
\end{lem}
\begin{proof}
Let $(L_i)_{i=0, \ldots, s}$ be the univariate Lagrange interpolation polynomials
at the values $f_{0}, \ldots, f_{s} \in \C$
and let $q_{i} (\xx) = L_{i} (f (\xx))$.

The polynomials $q_{i}$ are constructed so that
\begin{itemize}
 \item $q_{i} (V_{j}) = 0$ if $j\neq i$,
 \item $q_{i} (V_{i}) = 1$,
\end{itemize}
where $V_{i}= \Vc (J_{i})$. 
As the set  $\{f_{r+1},\ldots, f_{s}\}$ is stable by conjugation
and $f_{0}, \ldots, f_{r}\in \R$,  by construction of the Lagrange
interpolation polynomials we deduce that $q_{0}, \ldots, q_{r} \in \R[\xx]$.

By Hilbert's Nullstellensatz, there exists $N\in \N$ such that
$q_{i}^{N} \in \bigcap_{j\neq i} J_{j}$. 
As $\sum_{j=0}^{s} q_{j}^{N}=1$ on $V_{grad} $ and $q_{i}^{N}
q_{j}^{N} =0  \mod  \bigcap_{i} J_{i}= J_{grad }$ for $i\neq j$, we deduce that there
exists $N'\in \N$ such that
\begin{eqnarray*}
0 & = & ( 1-\sum_{j=0}^{s} q_{j}^{N})^{N'} \mod J_{grad} \\
  & = & 1 - \sum_{j=0}^{s} (1-(1-q_{j}^{N})^{N'}) \mod J_{grad} .\\
\end{eqnarray*}
As the polynomial $p_{i}= 1-(1-q_{j}^{N})^{N'}\in \C[\xx]$ is divisible by $q_{j}^{N}$,
it belongs to $\bigcap_{j\neq i} J_{j}$.
Since $q_{j}\in \R[\xx]$ for $j=0, \ldots, r$, we have $p_{j}\in \R[\xx]$
for $j=0, \ldots, r$, which ends the proof of this lemma.
\end{proof}

\begin{lem}\label{lem:preorder}
$-1 \in \Pc^{+}(C) + (\bigcap_{i>r} J_{i}^{\xx})$. 
\end{lem}
\begin{proof}
As $\bigcup_{i>r}\overline{\pi^{\xx} (V_{i})}^{\R} \cap \Sc^{+}(C) = \Vc^{\R}
(\bigcap_{i>r} J_{i}\cap \R[\xx])\cap \Sc^{+}(C) =  \Vc^{\R}
(\bigcap_{i>r} J_{i}^{\xx})\cap \Sc^{+}(C) = \emptyset$,
we have 
$\Ic (\Vc^{\R}(\bigcap_{i>r} J_{i}^{\xx})\cap \Sc^{+}(C))= \R[\xx] \ni 1$
and  by the Positivstellensatz (Theorem \ref{null} (iii)), 
$$ 
-1 \in \Pc^{+}(C) + (\bigcap_{i>r} J_{i}^{\xx}).
$$
\end{proof}

\begin{cor}\label{cor:Sminempty}
If $S_{min}=\emptyset$, then $-1\in \Pc^{+}(C)  + J_{grad}^{\xx}.$
\end{cor}
\begin{proof}
If $S_{min}=\emptyset$, then $f$ has no real KKT critical value on $S
(C)$ and
$r=-1$. Lemma \ref{lem:preorder} implies that  
$-1 \in \Pc^{+}(C)  + (\bigcap_{i=0}^s J_{i}^{\xx})=\Pc^{+}(C)  +
J_{grad}^{\xx}$. 
\end{proof}
In this case, 
 $\forall p\in \R[\xx]$, 
$p = \frac{1}{4} ((p+1)^{2} - (p-1)^{2})  \in \Pc^{+}(C)   +
J_{grad}^{\xx}$. If $C^{0}$ is chosen such that
$V(C^{0}) \subset V_{grad}^{\xx}$ then 
$S_{min}=\emptyset$ if and only if 
$-1 \in \Pc (C)$. 

We recall another useful result on the representation of positive
polynomials (see for instance \cite{DNP}):
\begin{lem}\label{epsilon}
Let $J \subset \R[\zz]$ and $V= \Vc (J)$ such that  $f (V)= f^*$ with
$f^{*} \in \R^{+}$
. There
exists $t\in \N$, s.t. $\forall \epsilon >0$, $\exists\, q \in \R[\xx]$
with $\deg (q) \leq t$ and $f+\epsilon = q^{2} \mod J$. 
\end{lem}
\begin{proof}
We know that $\frac{f+\epsilon}{f^*+\epsilon}-1$ vanishes on V. By
Hilbert's Nullstellensatz $(\frac{f+\epsilon}{f^*+\epsilon}-1)^l \in
J$ for some $l\in \N$.
From the binomial theorem, it follows that
\begin{equation*}
 (1+(\frac{f+\epsilon}{f^*+\epsilon}-1))^{1/2} \equiv \sum^{l-1}_i \binom{1/2}{k} (\frac{f+\epsilon}{f^*+\epsilon}-1)^k \stackrel{def}{=} \frac{q}{\sqrt{f^*+\epsilon}} \ mod \ J 
\end{equation*}
Then $f +\epsilon = q^2 \ mod \ J$.
\end{proof}

In particular, if $f^{*}>0$ this lemma implies that $f = (f-\frac{1}{2}f^{*}) + \frac{1}{2}f^{*} = q^{2} \mod J$ for some $q\in \R[\xx]$.

\begin{thm}\label{thm:sos+rad}
Let $C\subset \R[\xx]$ be a set of
constraints such that $\Sc^{+}(C)=\Sc^{+} (\gb)$, let $f \in \R[\xx]$, let $f_{0}< \cdots< f_{r}$ be the real $KKT$
critical values of $f$ on $S$ and let $p_{0}, \ldots, p_{r}$ be the
associated polynomials defined in  Lemma \ref{lem:lagrange}.
\begin{enumerate}
 \item $f -\sum_{i=0}^{r} f_{i}\, p_{i}^{2} \in \Pc^{+}(C) + \sqrt{J_{grad}^{\xx}}$. 
\item If $f\geq 0$ on $S_{KKT}$, then $f\in \Pc^{+}(C) + \sqrt{J_{grad}^{\xx}}$. 
 \item If $f>0$ on $S_{KKT}$, then $f\in \Pc^{+}(C) +
  J_{grad}^{\xx}$.
\end{enumerate} 
\end{thm}
\begin{proof}
By Lemma \ref{lem:lagrange}, we have 
$$ 
1 = (\sum_{i=0}^{s} p_{i})^{2} = \sum_{i=0}^{s} p_{i}^{2}  \mod J_{grad} .
$$
Thus $f = \sum_{i=0}^{s}  f \, p_{i}^{2} \mod J_{grad} $.

By Lemma \ref{lem:preorder}, $-1 \in \Pc^{+}(C)  + (\bigcap_{j>r} J_{j}^{\xx})$ so that
$f=\frac{1}{4} ( (f+1)^{2}- (f-1)^{2}) \in \Pc^{+}(C)  + \cap_{j>r}
J_{j}^{\xx}$ and 

\begin{equation}\label{eq:dec}
\sum_{i>r}f\, p_{i}^{2} \in \Pc^{+}(C)  + \bigcap_{j=0}^{s}
J_{j}^{\xx} = \Pc^{+}(C) + J^{\xx}_{grad}.
\end{equation} 

As the polynomial $(f-f_{i})\, p_{i}^{2}$ vanishes on $V_{grad} $, we
deduce that 
$$ 
f = \sum_{i=0}^{r} f_{i}\, p_{i}^{2} + \sum_{i=r+1}^{s} f\, p_{i}^{2}
+ \sqrt{J_{grad}^{\xx}}
= \sum_{i=0}^{r} f_{i}\, p_{i}^{2}  + \Pc^{+}(C) +  \sqrt{J_{grad}^{\xx} },
$$
which proves the first point.

If $f\geq 0$ on $S_{KKT}$, then $f_{i}\geq 0$ for $i=0, \ldots, r$ and $\sum_{i=0}^{r} f_{i}\,
p_{i}^{2} \in  \Pc^{+}(C) $ so that 
$$
f \in \Pc^{+}(C) + \sqrt{J_{grad}^{\xx}},
$$
which proves the second point.

If $f>0$ on $S_{KKT}$ by Lemma \ref{epsilon}, we have $f p_{i}^{2} =
q_{i}^{2} \mod J_{grad}^{\xx}$ with $q_{i} \in \R[\xx]$, which shows
that 
$$
\sum_{i=0}^{r} f_{i}\, p_{i}^{2}
=  
\sum_{i=0}^{r} q_{i}^{2} \mod J_{grad}^{\xx}
$$
Therefore, $\sum_{i=0}^{r} f_{i}\, p_{i}^{2} \in  \Pc^{+}(C) +
J_{grad}^{\xx}$ and $f\in \Pc^{+}(C) + J_{grad}^{\xx}$ by \eqref{eq:dec}, which proves the third point.
\end{proof}


This theorem involves only polynomials in $\R[\xx]$
and the points (2) and (3) generalize
results of \cite{DNP} on the representation of
positive polynomials.

Let us give now a refinement of Theorem \ref{thm:sos+rad} with a
control of the
degrees of the polynomials involved in the representation of $f$ as an
element of $\Pc^{+} (C) + J_{grad}^{\xx}$.
\begin{thm}\label{epsilonth} Let $C\subset \R[\xx]$ be a set of
constraints such that $\Vc(C^{0}) \subset V_{grad}^{\xx}$ and $\Sc^{+}(C)=\Sc^{+} (\gb)$.
If $f\geq 0$ on $S_{KKT}$, then there exists $t_{0}$ such that $\forall
\epsilon>0$, 
$$ 
f+\epsilon \in 
\Pc_{t_{0}} (C).
$$
\end{thm}
\begin{proof}
Let $J_{grad}= (C^{0}) \cap I_{grad}\subset \R[\zz]$, so that $\Vc
(J_{grad}) = V_{grad}$ since $\Vc(C^{0}) \subset V_{grad}^{\xx}$.
Using the decomposition \eqref{eq:dec} obtained in the proof of Theorem
\ref{thm:sos+rad},
we can choose $t'_{0} \in \N$ and $t_{0}\geq t_{0}' \in \N$ big enough
such that $\deg (p_{i})\leq t_{0}/2$ and 
$$ 
\sum_{i>r}f\, p_{i}^{2} \in \Pc _{t'_{0}}^{+}(C) + \
J_{grad}\cap \R[\xx]_{t'_{0}} \subset \Pc_{t_{0}} (C),
$$
since $J_{grad}^{\xx}= (C^{0}) \cap I_{grad}^{\xx}\subset (C^{0})$. 
Then $\forall \epsilon>0$,
\begin{equation}\label{eq:mod1}
\sum_{i>r} (f+\epsilon)\, p_{i}^{2} =
\sum_{i>r} f\, p_{i}^{2} 
+
\sum_{i>r} \epsilon\, p_{i}^{2} 
 \in \Pc_{t_{0}} (C).
\end{equation}

As $\forall \ \epsilon>0, \ f + \epsilon >0$ on $S_{KKT}$, i.e, $f_{i} +
\epsilon >0$ for $i=0, \ldots, r$, we deduce
from Lemma \ref{epsilon} that if $t_{0}$ is big enough, we have
\begin{equation}\label{eq:mod2}
(f +\epsilon)\, p_{i}^{2}= q_{i}^2 \mod \SpanDeg{C^{0}}{t_{0}}\cap \R[\xx]
\end{equation}
with $deg (q_{i}) \leq t_{0}/2$ for $i=0, \ldots, r$. 

Since $1 - \sum_{i=0}^{s} p_{i}^{2} = 0 \mod (C^{0})$, we can choose
$t_{0}$ big enough so that 
\begin{equation}\label{eq:mod3}
 (f+\epsilon) - \sum_{i=0}^{s} (f+\epsilon)\, p_{i}^{2}
 \in\SpanDeg{C^{0}}{t_{0}} \cap \R[\xx].
\end{equation}

From Equations \eqref{eq:mod1}, \eqref{eq:mod2}, \eqref{eq:mod3}, we deduce that if $t_{0}\in \N$ is big
enough, $\forall \epsilon>0$ 
$$ 
f+\epsilon \in \Pc_{t_{0}} (C),
$$
which concludes the proof of the theorem.
\end{proof}


\section{Finite convergence}

In this section, we show that the sequence of relaxation problems
attains its limit in a finite number of steps and that the minimizer
ideal can be recovered from an optimal solution of the corresponding relaxation problem.
%
%
We use the following notation:
\begin{itemize}
 \item $f^{*}= \inf_{\xx \in S_{KKT}} f (\xx)$
 \item $S_{min}=\{\xx^{*} \in S_{KKT}\mid f (\xx^{*}) =f^{*}\}$
\end{itemize}

We first show that $S_{min}=\emptyset$ can be detected from an adapted 
relaxation sequence:
\begin{prop}\label{prop:empty}
Let $C=(C^{0};C^{+})$ be a set of constraints of $\R[\xx]$, such that
$S_{min} \subset \Sc(C)$ and $\Vc(C^{0}) \subset V^{\xx}_{KKT}$ and
$C^{+}=\gb^{+}$. Then 
$S_{min}=\emptyset$, if and only if, there $t_0 \in \N$ such that $\forall t\ge t_0$, $\Lc_{t}(C)=\emptyset$.
\end{prop}
\begin{proof}
Let $J_{grad}= (C^{0})\cap I_{grad}$ and let $C'$ be a set of
constraints such that $(C'^{0})=J_{grad}\cap \R[\xx]=J^{\xx}_{grad}$ and $C'^{+}=\gb^{+}$ be a
finite set. By hypothesis, $\Vc (J_{grad})= V_{grad}$. We deduce from
Corollary \ref{cor:Sminempty} that if $S_{min}=\emptyset$, then 
$$ 
-1 \in \Pc^{+} (C') + (C'^{0}) \subset \Pc (C)= \cup_{t\in \N} \Pc_{t} (C).
$$
Thus there exists $t_{0}$ such that 
$-1 \in \Pc_{t} (C)$ for $t\geq t_{0}$, which
implies that $\Lc_{t} (C)=\emptyset$, since if there exists $\Lambda \in
\Lc_{t} (C)$, then $\Lambda (1)=1$ and $\Lambda (-1)\ge 0$.
 
Conversely, suppose that $S_{min} \neq \emptyset$ contains a point
$\xx^{*}$. As $S_{min}\subset \Sc (C)$, for all $t\in \N$ the evaluation  $\underline{\mathbf{1}}_{\xx^*}$  at $\xx^*$
restricted to $\R[\xx]_{2t}$ is an element of $\Lc_{t} (C)\neq \emptyset$.
\end{proof}
This proposition gives a way to check whether $S_{min}=\emptyset$, using
the relaxation sequence $\Lc_{t} (C)$. We are now going to analyse the case
where $f$ has KKT minimizers on $S$.

\textit{From now on, we assume that $S_{min}\neq \emptyset$}.
 
First, we recall a property similar to \cite[Claim 4.7]{LLR07}:
\begin{prop}\label{radical}
Let $C=(C^{0};C^{+})$ be a set of constraints of $\R[\xx]$. 
There exists $t_0 \in \N$ such that $\forall t\ge t_0$, 
 $\forall \Lambda \in \Lc_{t}(C),\  \sqrt[C^{+}]{C^{0}} \subset (\ker M_{\Lambda}^t)$.
\end{prop}

\begin{proof}
Let $C^{0}=\{g_{1}, \ldots, g_{l}\}$ and let $q_1,\ldots, q_k$ be
generators of $J := \sqrt[C^{+}]{C^{0}}$. 
By the Positivestllensatz, for $j \in {1,\ldots,k}$, there exist
$m_j \in \N^{*}$ and polynomials $u_r^{(j)}\in \R[\xx]$ and $\sigma_j\in \Pc^{+}(C)$
such that 
$$
q_j^{2m_j} + \sigma_j= \sum_{r=1}^l u_r^{(j)} g_r.
$$
Let us take $t_{0}\in \N$ big enough such that 
$u_r^{(j)}g_r\in \SpanDeg{C}{t_{0}}$ and $\sigma_{j} \in \Pc_{t_{0}}^{+}(C)$. 
Then for all $t\geq t_{0}$ and all  $\Lambda \in \Lc_{t}(C)$, we have $\Lambda
(u_r^{(j)}g_r)=0$, $\Lambda (q_j^{2m_j}) \geq 0$, $\Lambda (\sigma_j)
\geq 0$ and $\Lambda (q_j^{2m_j}) + \Lambda (\sigma_j) = 0$, which implies
that  $\Lambda (q_j^{2m_j})=0$ and $q_j \in \ker M_{\Lambda}^t$.
This proves that $(q_{1}, \ldots, q_{l}) = J \subset ( \ker M_{\Lambda}^t)$.
\end{proof}

\begin{remark}\label{rem:extend}
With the same arguments, we can show that for any $t'\in \N$, there
exists $t_0'\ge t'$ such that $\forall t\ge t_0'$, $\forall \Lambda \in
\Lc_{t}(C)$, 
$$ 
\SpanDeg{Q}{t'} \subset  \ker M_{\Lambda}^t,
$$
where $Q=\{q_{1}, \ldots, q_{k}\}$ generates  $J =
\sqrt[C^{+}]{C^{0}}$.

\end{remark}

The next result shows that in the sequence of optimization problems
that we consider, the minimum of $f$ on $S_{KKT}$ is reached from
some degree.


\begin{thm}\label{cmayor}\label{cpi-kernel}\label{mayor}\label{pi-kernel}
Let $C$ be a set of constraints of $\R[\xx]$ such that 
$S_{min} \subset \Sc (C) \subset V^{\xx,\R}_{KKT}$. There exists $t_1 \ge 0$ such that $\forall t\ge t_1$,
\begin{enumerate}
 \item $f^{\mu}_{t,C} = f^{*}$ is  reached for some $\Lambda^* \in
   \Lc_{t} (C)$,
\item $\forall \Lambda^* \in \Lc_{t} (C)$
with $\Lambda^{*} (f)=f^{\mu}_{t,C}=f^*$, we have $p_i \in  \ker M_{\Lambda^*}^t, \ \forall i=1,\ldots,r$,
 \item if $\Vc(C^{0}) \subset V^{\xx}_{KKT}$ then $f^{sos}_{t,C} = f^{\mu}_{t,C} = f^{*}$.
\end{enumerate}
\end{thm}
\begin{proof}
By Theorem \ref{thm:sos+rad}(1) applied to $f-f^{*}$, we can write 
 \begin{equation*}
  f-f^*\equiv \sum_{i=1}^r (f_i-f^*)\,p_i^2 + h +g. \label{desco3}
 \end{equation*}
with $h\in\Pc^{+} (C)$ and $g \in \sqrt{I_{grad}}\cap \R[\xx] =
\sqrt{I_{KKT}} \cap \R[\xx] \subset
\sqrt[\R]{I_{KKT}}\cap \R[\xx]$ (by Proposition \ref{prop:proj}). 
Since $\Sc (C)\subset V^{\xx,\R}_{KKT} =\overline{\pi^{\xx} (V^{\R}_{KKT})}$, we have 
$\sqrt[\R]{I_{KKT}}\cap \R[\xx] \subset \Ic (\Sc (C))=
\sqrt[C^{+}]{(C^{0})}$ 
 by
the Positivstellensatz. We deduce that $g \in \sqrt[C^{+}]{(C^{0})}$.
By proposition \ref{radical}, there exists $t_{1}\geq t_{0}$ such that
for all $t\ge t_{1}$, for all $\Lambda \in \Lc_{t}(C)$, 
 $\Lambda (g)=0$,  $\Lambda (h)\ge0$. 

Let us fix $t\geq t_{1}$ and $\Lambda^{*} \in \Lc_{t}(C)$ such that 
$\Lambda^{*} (f)=f^{\mu}_{t,C}$.
 Then
\begin{equation*}\label{desco4}
\Lambda^{*} (f-f^{*}) = \sum_{i=1}^r (f_i-f^*) \Lambda^{*} (p_i^2) + \Lambda^{*}(h).
\end{equation*} 
As $f_i - f^* = f_{i}-f_{0}> 0$, $\Lambda^*(p_i^2) \ge 0$ and
$\Lambda^{*}(h) \ge 0$ ($h\in \Pc_{t}^{+} (C)$), 
we deduce that $\Lambda^{*} (f-f^{*}) = \Lambda^{*} (f)-f^{*} \ge 0$.

As $\emptyset \neq S_{min}\subset \Sc (C)$, we have 
$\Lambda^{*} (f)\leq f^{*}$ (by Remark \ref{rem:fmin}), 
so that $\Lambda^{*} (f) = f^{\mu}_{t,C} =f^{*}$, which proves the
first point.
Hence for $i=1, \ldots, r$, $\Lambda^{*} (p_i^2)=0$ and $p_{i} \in \ker
M_{\Lambda^{*}}^{t}$, which proves the second
point.

To prove that $f^{sos}_{t,C} =f^{*}$ when  $\Vc (C^{0}) \subset V_{KKT}^{\xx}$, we apply Theorem \ref{epsilonth} to
$f-f^*$ which is positive on $S_{KKT}$. Let us take $J_{grad} = (C^{0}) \cap
I_{grad} \subset \R[\zz]$. We denote by $\tilde{C}$ the set of
constraints such that $\tilde{C}^{0}$ is a finite
family of generators of $J_{grad}\cap \R[\xx]$ and $\tilde{C}^{+}=C^{+}$.

By Theorem \ref{epsilonth}, there exists $t_{0}$ such that $\forall \epsilon>0$,
$$
f-f^*+\epsilon  \in \Pc_{t_{0}} (\tilde{C}).
$$
As $(\tilde{C}^{0}) = (C^{0}) \cap I_{grad} \subset (C^{0})$, we can choose $t_{1}\ge t_{0}$ such that 
$\SpanDeg{\tilde{C}}{t_{0}} \subset \SpanDeg{C}{t_{1}}$ and 
$\Pc_{t_{0}} (\tilde{C}) \subset \Pc_{t_{1}} (C)$. 

Then $\forall t\geq t_{1}$, 
$f-f^*+\epsilon \in \Pc_{t} (C)$. Hence by
maximality, $\forall \epsilon >0, f^*-\epsilon \le f^{sos}_{t,C}$.
We deduce that $f^*\le f^{sos}_{t,C}$, which implies that
$f^{sos}_{t,C} = f^{\mu}_{t,C} = f^{*}$
and proves the third point.
\end{proof}

As for the construction of generators of 
$\sqrt[C^{+}]{I_{KKT}}$ (Proposition \ref{radical}), we can
construct generators of $I_{min}$ from the kernel of a
truncated Hankel operator associated to any linear form which
minimizes $f$, using the following propositions:

\begin{prop}\label{gminequal}
$I_{min}=(p_1,...,p_r) + \sqrt[C^{+}]{I_{KKT}^{\xx}}$.
\end{prop}
\begin{proof}
First of all, we proof that $I_{min}^{\zz}=(p_1,...,p_r) + \sqrt[C^{+}]{I_{grad}}= (p_1,...,p_r) + \sqrt[\R]{I_{grad}}$.\\
Using the decomposition of Lemma \ref{pis} and the polynomials $p_{i}$
of Lemma \ref{lem:lagrange}, we have
$$
V^{\R}_{{grad}} = (V_0 \cup V_1 \cup \cdots \cup V_s) \cap \R^{n+n_1+2\,n_2} = V_0^{\R}\cup \cdots \cup V_r^{\R},
$$ 
By construction, $\Ic(V_0^{\R})=I_{min}^{\zz}$, $p_i(V_0^{\R})=0$ for $i=1,\ldots, s$  and 
$p_i \in \R[\xx]$ for $i=0,\ldots,r$. This implies that $p_{i}\in
I_{min}^{\zz}$ for $i=1, \ldots, r$.

As $V_{0}^{\R} \subset V_{grad}^{\R}$, we also have $\sqrt[C^{+}]{I_{grad}} \subset I_{min}^{\zz}$.

We have proved so far that $(p_{1}, \ldots , p_{r}) + \sqrt[C^{+}]{I_{grad}} \subset I_{min}^{\zz}$.
In order to prove the reverse inclusion, we denote by $q_{1}, \ldots, q_{m}$ a  family of generators of the ideal $I_{min}^{\zz}$. Take one of these
 generators $q_{j}$ ($1 \le j \le m$).
By construction, $q_j\,p_0 ( {V_0^{\R}} ) =0$ and 
 $ q_jp_0 ( {V_i^{\R}} ) =0 $  for $i=1,\ldots,r$, which implies that $q_jp_0 \in \sqrt[C^{+}]{I_{grad}}$. 

By Lemma \ref{lem:lagrange}, we have the decomposition
$$
q_j \equiv q_j(p_0 +p_1+\cdots+p_s) \mod I_{grad} \subset \sqrt[C^{+}]{I_{grad}}.
$$ 
Moreover  $(p_{r+1}+ \cdots + p_{s}) \in \R[\zz]$ and vanishes on
$V_{k}^{\R}$ for $k=0, \ldots, r$. Thus $(p_{r+1}+ \cdots + p_{s}) \in
\sqrt[C^{+}]{I_{grad}}$ and we deduce that $q_{j} \in (p_{1}, \ldots, p_{r}) + \sqrt[C^{+}]{I_{grad}}$. This
proves the other inclusion and the first equality.

As $V^{\R}_{grad}= V_{grad}^{\R} \cap \Sc^{+} (C)$ (Remark \ref{rem:proj}),
by the Positivstellensatz,  $\sqrt[C^{+}]{I_{grad}}=
\sqrt[\R]{I_{grad}}$, which proves the second equality. 

By the Positivstellensatz and Remark \ref{rem:proj},  we have
$$
\sqrt[C^{+}]{I_{grad}}\cap \R[\xx]=\sqrt[\R]{I_{grad}} \cap \R[\xx]
=\Ic(\pi^{\xx}(V^{\R}_{grad}))=\Ic(\pi^{\xx}(V_{KKT})^{\R} \cap
\Sc^{+}(C))=\sqrt[C^{+}]{I_{KKT}^{\xx}}.
$$
and
$$ 
I_{min}
=I_{min}^{\zz} \cap \R[\xx]=(p_1,...,p_r)\cap \R[\xx] + \sqrt[C^{+}]{I_{grad}}\cap
\R[\xx] 
= (p_1,...,p_r)+\sqrt[C^{+}]{I_{KKT}^{\xx}}.
$$
which proves the equality.
\end{proof}




\begin{thm}\label{cminker}
 For $C\subset \R[\xx]$ with $ S_{min} \subset \Sc (C) \subset V_{KKT}^{\xx,\R}$,
there exists $t_2 \in \N$ such that
$\forall t\ge t_2$, for $\Lambda^* \in \Lc_{t} (C)$
with $\Lambda^{*} =f^{\mu}_{t,C}$, we have $I_{min} \subset (\ker M_{\Lambda^*}^t)$.
\end{thm}

\begin{proof}
 To prove the inclusion we take $t_2=\max\{t_0,t_1\}$ and we combine
 Proposition \ref{gminequal} with Proposition \ref{radical} for $C\subset \R[\xx]$ and Theorem \ref{cpi-kernel}.
\end{proof}

We introduce now the notion of {\em optimal linear form for $f$}. Such
a linear form allows us to compute $I_{min}$.

\begin{prop}\label{generic}
 For $\Lambda^* \in \Lc_{t} (C)$ and $p \in \R[\xx]$, the following assertions are equivalent:
 \begin{itemize}
  \item[(i)] $\rank M_{\Lambda^*}^{t} = \max_{\Lambda \in \Lc_{t} (C),\Lambda(p)=p^{\mu}_{t,C}} \rank M_{\Lambda}^{t}$.
  \item[(ii)]  $\forall \Lambda \in \Lc_{t} (C)$
    such that $\Lambda(p)=p^{\mu}_{t,C}$, $\ker M_{\Lambda^*}^{t} \subset \ker M_{\Lambda}^{t}$.
 \end{itemize}
 We say that  $\Lambda^* \in \Lc_{t} (C)$ is
 optimal for $p$ if it satisfies one of the equivalent conditions (i)-(ii).
 \end{prop}
A proof of this proposition can be found in \cite{lasserre:hal-00651759}(Proposition 4.7).
 
\begin{remark}
A linear form $\Lambda^* \in \Lc_{t} (C)$
 optimal for $p$ can be computed by solving a Semi-Definite
 Programming problem by an interior point method
 \cite{Lasserre:book}. In this case, the solution $\Lambda^{*}$
 obtained by convex optimization is in the interior of the face
of linear forms that minimize $f$.
\end{remark}

The next result, which refines Theorem \ref{cminker}, shows that only
elements in $I_{min}$ are involved in the kernel of a truncated
Hankel operator associated to an optimal linear form for $f$.

\begin{thm}\label{ckermin}
 Let $t \in \N$ such that $f \in \R[\xx]_{2t}$ and let $C \subset \R[\xx]_{2t}$  with $S_{min}\subset
\Sc (C)$. If $\Lambda^* \in \Lc_{t} (C)$ 
is optimal for $f$ and such that $\Lambda^{*} (f) = f^{*}$, then $\ker
M_{\Lambda^*}^{t} \subset I_{min}$.  
\end{thm}

\begin{proof}
It is similar to proof of Theorem 4.9 in \cite{lasserre:hal-00651759}.
\end{proof}

The last result of this section shows that an optimal linear form for $f$
yields the generators of the minimizer ideal $I_{min}$ in high enough degree.

\begin{thm}\label{cthprin}
Let $\gb \subset \R[\xx]$ be a set of constraints with $S_{min}\neq \emptyset$.
 For a set of constraints $C\subset \R[\xx]$ with $S_{min} \subset \Sc (C) \subset V_{KKT}^{\xx,\R}$,
there exists $t_2\in \N$ (defined in Theorem \ref{cminker}) such that
$\forall t\ge t_2$, 
\begin{itemize}
\item $f^{\mu}_{t,C}= \min _{\xx\in S_{KKT}} f (\xx) $ is reached for some $\Lambda^* \in
  \Lc_{t} (C)$, 
 \item $\forall \Lambda^* \in \Lc_{t} (C)$ optimal for $f$, we have $\Lambda^*(f)=f^*$ and
$(\ker M_{\Lambda^*}^t)=I_{min}$,
 \item if $\Vc(C^{0}) \subset V_{KKT}^{\xx}$ then $f^{sos}_{t,C} = f^{\mu}_{t,C} = f^{*}$.
\end{itemize} 
\end{thm}
\begin{proof}
We obtain the result as a consequence of Theorem \ref{cmayor}, Theorem \ref{cminker}
and Theorem \ref{ckermin}.
\end{proof}
The same results hold if we replace $C$ by any other finite set
defining a {\em  real} variety such that $S_{min} \subset \Sc (C) \subset V_{KKT}^{\xx,\R}$.
\begin{remark}\label{rem:consequence}
We can also replace the initial set of constraints $\gb$ by any other
set  $\tilde{\gb}$ defining the same semi-algebraic set $S=\Sc (\gb)=
\Sc (\tilde{\gb})$ and consider the KKT variety associated to $\tilde{\gb}$.
\end{remark}


 
\section{Consequences}

Let us describe now some consequences of these results in specific
cases, which have been previously studied.

\subsection{Global optimization}
We consider here the case $n_{1}=n_{2}=0$. Theorem
\ref{thm:sos+rad} implies the following result (compare with \cite{NDS}):
\begin{thm}
Let $f \in \R[\xx]$. 
\begin{enumerate}
\item If the real critical values of $f$ are positive, then $f\in \Qc^{+} +  \sqrt{(\frac{\partial f}{\partial_{x_{1}}}, \ldots, \frac{\partial f}{\partial_{x_{n}}})}$. 
\item If the real critical values of $f$ are strictly positive, then $f\in \Qc^{+} +  (\frac{\partial f}{\partial_{x_{1}}}, \ldots, \frac{\partial f}{\partial_{x_{n}}})$.
\end{enumerate} 
\end{thm} 
In particular, if there is no real critical value, then $f \in \Qc^{+}
+  (\frac{\partial f}{\partial_{x_{1}}}, \ldots, \frac{\partial
  f}{\partial_{x_{n}}})$. 
 
A consequence of Proposition \ref{prop:empty} and Theorem \ref{cthprin} is the following: 
\begin{thm}
Let $f \in \R[\xx]$ and $C=\{\frac{\partial f}{\partial_{x_{1}}}, \ldots, \frac{\partial f}{\partial_{x_{n}}} \}$. Then, there exists $t_{0}\in \N$, such that $\forall
t\geq t_{0}$  either $\Lc_{t} (C)=\emptyset$ and $S_{min}=\emptyset$ or
\begin{enumerate}
 \item $f^{sos}_{t,C} = f^{\mu}_{t,C} = f^{*} = \min_{\xx \in \R^{n}}
   f (\xx)$ is reached for some $\Lambda^* \in \Lc_{t} (C)$,
 \item $\forall \Lambda^* \in \Lc_{t} (C)$
optimal for $f$, $\ker M_{\Lambda^*}^t$ generates $I_{min}$.
\end{enumerate}
\end{thm}
The first point of this theorem can also be found in \cite{NDS}. 

\subsection{General case}
 
A direct consequence of Proposition \ref{prop:empty} and Theorem \ref{cthprin} is the following:
\begin{thm}\label{thm:exact}
Let $C\subset \R[\xx]$ be a set of constraints such that 
\begin{itemize}
 \item $(C^{0})= I_{KKT} \cap \R[\xx]$,
 \item $C^{+}=\gb^{+}$.
\end{itemize}
Then there exists $t_0\in \N$ such that $\forall t\ge t_0$, either
$\Lc_{t} (C)=\emptyset$ and $S_{min}=\emptyset$ or
\begin{itemize}
\item $f^{sos}_{t,C}=f^{\mu}_{t,C}= min _{\xx \in S_{KKT}} f (\xx)$ is reached for some $\Lambda^* \in \Lc_{t} (C)$, 
 \item $\forall \Lambda^* \in \Lc_{t} (C)$ optimal for $f$, we have $\Lambda^*(f)=f^*$ and
$(\ker M_{\Lambda^*}^t)=I_{min}$. 
\end{itemize} 
\end{thm}
The set $C^{0}$ is constructed so that $\Vc (C^{0}) =
V_{KKT}^{\xx}$. As we have seen, the weaker condition $S_{min}\subset
\Sc(C) \subset V_{KKT}^{\xx}$
is sufficient to have an exact relaxation sequence. 

The generators $C^{0}$ of $I_{KKT}\cap \R[\xx]$ can be computed by 
elimination techniques (for instance by Groebner basis computation with 
a product order on monomials \cite{CLO97}).

\subsection{Regular case} 
We consider here a semi-algebraic set $S$ such that its defining constraints
intersect properly.
For any $\xx\in \C^{n}$, let $\nu (\xx)=\{ j\in [1, n_{2}] \mid g_{j}^{+} (\xx) =0 \}$.

\begin{defn}
We say that a set of constraints $\gb=(g_{1}^{0}, \ldots g_{n_{1}}^{0}$; $g_{1}^{+}, \ldots, g_{n_{2}}^{+})$ is regular if for all points $\xx\in \Sc (\gb)$
with $\nu (\xx)=\{j_{1}, \ldots, j_{k}\}$, the vectors $\nabla g_{1}^{0} (\xx), \ldots, \nabla g_{n_{1}}^{0} (\xx)$,
$\nabla g_{j_{1}}^{+} (\xx), \ldots$, $\nabla g_{j_{k}}^{+} (\xx)$
are linearly independent.
\end{defn}
This condition is used for instance in \cite{Ha-Pham:10}. It implies that $\forall \xx \in S$, $|\nu (\xx)| \le n-n_{1}$ and that $B_{\nu
  (\xx)} (\xx)$ is of rank $n_{1} + |\nu (\xx)|$. 
A stronger condition, called the $\C$-regularity, corresponds to sets
of constraints such that  $\forall \xx \in \C^{n}$, $B_{\nu (\xx)} (\xx)$ is
of rank $n_{1} + |\nu (\xx)|$.
This condition is used for instance in \cite{Nie11}. It is satisfied
for semi-algebraic sets defined by ``generic'' constraints when $n_{1}\le n$ as shown in \cite{Nie11}.

If $\gb$ is regular, then for all points $\xx$ in $S$ the rank of
$B_{\nu (\xx)} (\xx)$ is $n_{1}+|\nu (\xx)|$ and 
$S_{sing}=\emptyset$. The decomposition \eqref{eq:decomp} implies that
$S_{FJ}=S_{KKT}$ and that all minimizer points of $f$ on $S$ are KKT
points. 
If moreover $\gb$ is $\C$-regular, then $V_{FJ}^{\xx}= \Vc
( \Gamma_{FJ}) = V_{KKT}^{\xx}$.

We deduce from Theorem \ref{cthprin} the following  result:
\begin{thm}
Let $\gb\subset \R[\xx]$ be a regular set of constraints and let $C\subset \R[\xx]$ be the set of constraints such that
\begin{itemize}
 \item $C^{0}=\Gamma_{FJ}$ defined in \eqref{eq:kktproj} (resp. $C^{0}=\Phi_{FJ}$ defined in \eqref{eq:kktprojreal}),
 \item $C^{+}=\gb^{+}$.
\end{itemize}
Suppose that $\min_{\xx \in \Sc (\gb)} f (\xx)$ is reached
at some point of $\Sc (\gb)$. 
Then, there exists $t_{0}\in \N$ such that $\forall t\ge t_{0}$,
\begin{enumerate}
 \item $f^{\mu}_{t,C} = f^{*}= \min_{\xx \in \Sc (\gb)} f (\xx)$ is reached for some $\Lambda^* \in \Lc_{t} (C)$,
 \item $\forall \Lambda^* \in \Lc_{t} (C)$ optimal for $f$, $\ker M_{\Lambda^*}^t$
generates $I^{\xx}_{min}$,
  \item If $\gb$ is $\C$-regular and $C^{0}=\Gamma_{FJ}$, then $f^{sos}_{t,C}=f^{\mu}_{t,C} = f^{*}$.
\end{enumerate}

\end{thm}
By Lemma \ref{lem:kktproj} and Remark \ref{rem:realkkt}, $C$ is 
constructed so that $S_{min} \subset \Sc (C) =S_{KKT} \subset V_{KKT}^{\xx,\R}$. 

Points (1) and (3) are proved for $C^{0}=  \Gamma_{FJ}$ in
\cite{Nie11} under the condition that $\gb$ is
$\C$-regular. These points can also be found in
\cite{Ha-Pham:10} for  $C^{0}= \gb^0 \cup \Phi_{FJ}$ under the
condition that $\gb$ regular (but a problem appears in the proof: the vanishing of the polynomials $\Phi_{FJ}$ at a point $\xx \in \C^{n}$ does not imply
that $\rank\, A_{\nu (\xx)} (\xx) < n_{1} + |\nu (\xx)|$).

In this case, the relaxation constructed with $\Gamma_{FJ}$ (or $\Phi_{FJ}$) is exact and can be used to compute the minimizer ideals of $f$ on the
semi-algebraic set $S$.

\subsection{Zero dimensional real variety}\label{sec:zerodim}
Let $\gb \subset \R[\xx]$ be a set of constraints such that
$\Vc^{\R}(\gb^{0})$ is finite and let $S:= \Sc (\gb)$. By remark
\ref{rem:consequence}, we can assume that $S$ is defined by a set of
constraints $\tilde{\gb}$ such that $(\tilde{\gb}^{0})$ is radical. Then 
$\forall \xx\in \Vc (\gb^{0})=\Vc (\tilde{\gb}^{0})$, the Jacobian
matrix $\tilde{B}_{\nu (\xx)} (\xx)$ associated to $\tilde{\gb}^{0}$ is of rank $n$. Therefore
we have $\Vc (\gb^{0}) = \Vc (\tilde{\gb}^{0}) = V_{KKT}^{\xx}$ 
 and any point of $S$ is a $KKT$-point: $S=S_{FJ}=S_{KKT}$.
Consequently, we deduce from Theorem \ref{cthprin} the following result:
\begin{thm}
Let $\gb=(\gb^0,\gb^+) \subset \R[\xx]$ be a set of constraints such that 
$\Vc^{\R} (\gb^{0})$  is finite.
Then there exists $t_{0}\in \N$ such that $\forall t\ge t_{0}$,
\begin{enumerate}
 \item $f^{sos}_{t,\gb}=f^{\mu}_{t,\gb} = f^{*}= \min_{\xx \in \Sc (\gb)} f (\xx)$
is reached for some $\Lambda^* \in \Lc_{t} (\gb)$,
\item $\forall \Lambda^* \in \Lc_{t} (\gb)$ optimal for $f$, $\ker
  M_{\Lambda^*}^t$ generates $I_{min}$.
\end{enumerate}
\end{thm}
This answers an open question in \cite{LauSur}. The first point was also 
solved in \cite{NieReal} using dedicated techniques.



\subsection{Smooth real variety}
We consider a set of constraints $\gb=\{g_{1}^{0}, \ldots,
g_{n_{1}}^{0} \} \subset \R[\xx]$ such that $\Vc^{\R}(\gb^{0})$ is
equidimensional smooth and $\gb^{+}=\emptyset$.  This means that $S=
\Sc (\gb)= \Vc^{\R} (\gb^{0})$ is the union of irreducible components
of the same dimension $d$ and that for any point $\xx\in S$,
$B_{\emptyset} (\xx) =[\nabla g_{1}^{0} (\xx), \ldots, \nabla
g_{n_{1}}^{0}(\xx) ]$ is of rank $m=\dim S=n-d$.  
Therefore, $S_{sing}=\emptyset$.
In this case, $\nabla f (\xx)$ is a linear combination of $\nabla g_{1}^{0} (\xx),
\ldots, \nabla g_{n_{1}}^{0} (\xx)$, if and only if, $\rank
A_{\emptyset} (\xx)\le r$. 

The set $\Gamma_{FJ}$ defined in \eqref{eq:kktproj} (or $C^{0}=
\gb^{0} \cup \Phi_{FJ}$ defined in \eqref{eq:kktprojreal}),
or the union $\Delta^{n-d}$ of $\gb^{0}$ and the set of $(n-d+1)\times (n-d+1)$ minors of
the Jacobian matrix of $\{f,g_{1}^{0}, \ldots, g_{n_{1}}^{0}\}$, which
contain the first column $\nabla f$ define the variety $S_{KKT}$.

We deduce from  Theorem \ref{cthprin}, the following result:
\begin{thm}
Let $\gb=\{g_{1}^{0}, \ldots, g_{n_{1}}^{0}\} \subset  \R[\xx]$ such that
$S =\Vc^{\R}(\gb)$ is an equidimensional and smooth variety of dimension
$d$. 

Let $C \subset \R[\xx]$ be the set of
constraints such that $C^{0} = \Gamma_{FJ}$ defined in \eqref{eq:kktproj} (or $C^{0}=
\Phi_{FJ}$ defined in \eqref{eq:kktprojreal}, $C^{0}= \Delta^{n-d}$)
Then there exists $t_{0}\in \N$ such that $\forall t\ge t_{0}$, either
$\Lc_{t} (C)=\emptyset$ and $S_{min}=\emptyset$ or
\begin{enumerate}
  \item $f^{\mu}_{t,C} = f^{*}= \min_{\xx \in S} f (\xx)$
 is reached for some $\Lambda^* \in \Lc_{t} (C)$,
  \item $\forall \Lambda^* \in \Lc_{t} (C)$ optimal for $f$, $\ker
    M_{\Lambda^*}^t$ generates $I_{min}$. 
\end{enumerate}
\end{thm}

\subsection{Known minimum}
In the case where we know the minimum $f^*$ of $f$ on the basic closed
semi-algebraic set $S$, we take $\gb'$ with
$\gb'^{0}=\{\gb^0,f-f^*\}$ and $\gb'^{+}=\gb^{+}$. Let $S=\Sc (\gb)$, $S'=\Sc
(\gb')$. By construction $S_{min}\subset S'$ and
$S'=S'_{KKT}$ and $\Vc (\gb^{0})\subset V^{\xx}_{KKT} (\gb'^{0})$. Theorem \ref{cthprin} applied to $\gb'$ implies the following result:

\begin{thm}
Let $\gb=\{g_{1}^{0}, \ldots, g_{n_{1}}^{0};g_{1}^{+}, \ldots, g_{n_{2}}^{+}\} \subset  \R[\xx]$.
Let $f^*$ be the minimum of $f$ and $C\subset \R[\xx]$ the set of constraints such that $C^{0}=\{\gb^0,f-f^*\}$ and $C^{+}=\gb^{+}$.
Then there exists $t_{0}\in \N$ such that $\forall t\ge t_{0}$,
\begin{enumerate}
 \item $f^{sos}_{t,C}=f^{\mu}_{t,C} = f^{*}= \min_{\xx \in \Sc (C)} f (\xx)$
is reached for some $\Lambda^* \in \Lc_{t,C}$,
 \item $\forall \Lambda^* \in \Lc_{t}(C)$ optimal for $f$, $\ker M_{\Lambda^*}^t$
generates $I_{min}$.
\end{enumerate}
\end{thm}

\subsection{$\gb^{+}$-radical computation.}
In the case where $f=0$, by Remark \ref{rem:f=0} all the points of $S$
are KKT points and minimizers of $f$ so that $S_{min} = S = S_{KKT}$.
Moreover, $I_{KKT}^{\xx} = (g_{1}^{0}, \ldots, g_{n_{1}}^{0})$ since
$F_{1}, \ldots, F_{n}$, $v_{1} g_{1}^{+}, \ldots, v_{n_{2}} g_{n_2}^{+}$ are
homogeneous of degree $1$ in the variables $u_{1}, \ldots, u_{n_{1}},
v_{1}, \ldots, v_{n_{2}}$.
We deduce the following result:
\begin{thm}
Let $\gb=\{g_{1}^{0}, \ldots, g_{n_{1}}^{0};g_{1}^{+}, \ldots, g_{n_{2}}^{+}\} \subset  \R[\xx]$.
There exists $t_{2}\in \N$ such that $\forall t \geq t_{2}$, $\forall \Lambda^* \in \Lc_{t}(\gb)$ optimal for $0$, we have $(\ker M_{\Lambda^*}^t) = \Ic(\Sc)=\sqrt[\gb^{+}]{(\gb^{0})}$.
\end{thm}
 
This gives a way to compute $\sqrt[C^{+}]{(C^{0})}$ (see also \cite{Zhi}), which generalizes the
approach of \cite{LLR08b}, \cite{lasserre:hal-00651759} or \cite{Ros2009} to compute the real radical of an ideal.

\section{Examples}
This section contains examples that illustrate different aspects of
our method.
In the case of a finite number of minimizers for a function $f$ on the
semi-algebraic set $S$ defined by the set of constraints $\gb$, 
the approach we describe leads to the following algorithm:
\begin{enumerate}
 \item Compute $C\subset \R[\xx]$ such that $C^{0}$ generates $I_{KKT}\cap \R[\xx]$
   and $C^{+}=\gb^{+}$;
 \item $t:=\lceil\frac{1}{2}\max\{\deg (f), \deg (g^{0}_{i}), \deg (g^{+}_{j})\}\rceil$;
\item Compute $\Lambda^{*}\in \Lc_{t}(C)$ optimal for $f$ (solving a finite dimensional SDP problem by an interior point method);
 \item Check the convergence certificate for $M^{t}_{\Lambda^{*}}$
   (by flat extension \cite{HeLa05, MoLa2008});
 \item If it is not satisfied, then $t:=t+1$ and repeat from step (2);
 \item[]Otherwise compute $ K := \ker M^{t}_{\Lambda^{*}}$.
\end{enumerate}
Output $f^{*} = \Lambda^{*} (f)$ and the generators $K$ of $I_{min}$.

%

%
%

\begin{example} We consider the ``ill-posed'' problem 
$$ \min \ x\ s.t \ x^{3} \ge 0. $$
The ideal $I_{KKT}$ is $I_{KKT}= (1-3v_{1} x^{2}, v_{1} x^{3})=
(1)$. Thus $V_{KKT}=\emptyset$. According to the decomposition
\eqref{eq:decomp}, $S_{FJ}=S_{sing}$ and we compute the minimum of $x$
on $S_{sing}$, which is defined by $x^{2}=0$:
$$ \min \ x\ s.t \ x^{2}=0. $$
Now according to section \ref{sec:zerodim}, the relaxation associated
to this problem is exact and yields the solution $x=0$.
\end{example}

\begin{example}
We consider the following problem
$$\begin{array}{rl}
\min & f(x,y,z)=x^{2}+y^{2}+z^{2};\\
s.t & \rank \left ( \begin{array}{ccc}
x+z+1 & x+y & y+z \\
  x+y & y+z & x+z+1 
\end{array}
 \right) \leq 1
\end{array}
$$ 
or equivalently
$$\begin{array}{rl}
\min & f(x,y,z)=x^{2}+y^{2}+z^{2};\\
s.t & (x+z+1)(y+z)-(x+y)^2=0;\\
&(x+z+1)^2-(y+z)(x+y)=0;\\
&(x+z+1)(x+y)-(y+z)^2=0;\\
\end{array}
$$ 
This corresponds to computing the closest point on a twisted cubic defined
by $2\times 2$ minors.
The set of constraints $\gb$ is not regular but
$\Sc(\gb)=\Vc^{\R}(\gb^0)$ is a smooth real variety.

In the first iteration of the algorithm, the order is 1, the size of the Hankel matrix
$M_{\Lambda}^1$ is 3, $\min \Lambda(f) =1$ and there is no duality gap.
The flat extension condition is satisfied for $M_{\Lambda}^1$ and thus
we have found the minimum.
The algorithm stops and we obtain  $I_{min}=(x,y-1,z)$.
The points that minimize f are $\{(x=0,y=1,z=0)\}$.
\end{example}
 
\begin{example}
We consider the Motzkin polynomial,
\begin{equation*}
 \min \ f(x,y)=1+x^4y^2+x^2y^4-3x^2y^2
\end{equation*}
which is non negative on $\R^2$ but not a sum of squares in $\R[x,y]$.
 We compute its gradient ideal,
 $I_{grad}(f)=(-6xy^2+2xy^4+4x^3y^2,-6yx^2+2yx^4+4y^3x^2)$, which is
 not zero-dimensional.
 
 In the first iteration of the algorithm, the order is 3, the size of the Hankel
 matrix $M_{\Lambda}^3$ is 10, $\min \Lambda(f) =-216$. The flat extension condition is not satisfied hence we try with degree 4.
 
  In the second iteration the order is 4, the size of the Hankel
 matrix $M_{\Lambda}^4$ is 15, $\min \Lambda(f) =0$, there is no
 duality gap. The flat extension condition is satisfied for $M_{\Lambda}^4$ and we have found the minimum.
 The algorithm stops and we obtain $I_{min}=(x^2-1,y^2-1)$. 
 The points that minimize f are $\{(x=1,y=1),(x=1,y=-1),(x=-1,y=1),(x=-1,y=-1)\}$.
 
 For this example Gloptipoly must go until order 9 in order to satisfy the flat extension condition.
    

\end{example}    
    
\begin{example}
We consider the Robinson polynomial
 $$\min \ f(x,y)=1+x^6-x^4-x^2+y^6-y^4-y^2-x^4y^2-x^2y^4+3x^2y^2$$
which is non negative on $\R^2$ but not a sum of squares in $\R[x,y]$.
 We compute its gradient ideal, \\
$I_{grad}(f)=(6x^5-4x^3-2x-4x^3y^2-2xy^4+6xy^2,6y^5-4y^3-2y-4y^3x^2-2yx^4+6yx^2)$
which is not zero-dimensional.

In the first iteration, the order is 3, the size of the Hankel matrix $M_{\Lambda}^3$ is 10, $\min \Lambda(f) =-0.93$. 
The flat extension condition is not satisfied hence we try with degree 4.

In the second iteration the degree is 4, the size of the Hankel matrix $M_{\Lambda}^4$ is 15, $\min \Lambda(f) =0$.
There is no duality gap. The flat extension condition is satisfied for $M_{\Lambda}^3$
and we have found the minimum.

The algorithm stops and we obtain  $I_{min}=(x^3-x,y^3-y,x^2y^2-x^2-y^2+1)$.
The points that minimize $f$ are $\{(x=1,y=1),(x=1,y=-1),(x=-1,y=1), (x=-1,y=-1),
(x=1,y=0),(x=-1,y=0),(x=0,y=1),(x=0,y=-1)\}$.
    
For this example, Gloptipoly must go until order 7 in order to satisfy
the flat extension condition. 
\end{example}
 
\begin{example}
We consider the homogeneous Motzkin polynomial with a perturbation $\epsilon = 0.005$,
$$\begin{array}{rl}
\min & f(x,y,z)=x^4y^2+x^2y^4-3x^2y^2z^2+z^6+\epsilon(x^2+y^2+z^2);\\
s.t & h(x,y,z) = 1-x^2-y^2-z^2  \geq 0 
\end{array}
$$
This example coming from \cite[Example 6.25]{LauSur} is a case where the constraints $\gb$ define a compact semi-algebraic set, but the
direct relaxation using the associated quadratic module or preordering is not exact.

We add the projection of the KKT ideal and we have the similar problem
$$\begin{array}{rl}
\min & x^4y^2+x^2y^4-3x^2y^2z^2+z^6+0.005(x^2+y^2+z^2);\\
s.t & -4zx^4y-20zx^2y^3+12x^2yz^3-0.06zy^5+12.06yz^5=0;\\
&-20zx^3y^2-4zxy^4+12xy^2z^3-0.06zx^5+12.06xz^5=0;\\
&(4x^3y^2+2xy^4-6xy^2z^2+0.03x^5)(-x^2-y^2-z^2+1)=0;\\
&(2x^4y+4x^2y^3-6x^2yz^2+0.03y^5)(-x^2-y^2-z^2+1)=0;\\
&(-6x^2y^2z+6.03z^5)(-x^2-y^2-z^2+1)=0;
\end{array}
$$
where the first three equations are the $2\times 2$ minors of the
Jacobian matrix of $f$ and $h$
and the last three equations are the gradient ideal of $f$
multiplied by $h$.\\

In the first iteration the order is 5, the size of the Hankel matrix
$M_{\Lambda}^5$ is 167, $\min \Lambda(f) =0$, there is no duality gap.
The flat extension condition is satisfied for $M_{\Lambda}^5$ and we have found the minimum.
The algorithm stops and we obtain  $I_{min}=(x,y,z)$.
The point that minimize $f$ is $(0,0,0)$.

For this example, the flat extension condition does not hold with Gloptipoly if $\epsilon \le 0.01$. 
\end{example}  

Finally with these two last examples we show that even the minimizer
ideal $I_{min}$ is not zero-dimensional we can recover it from a solution of the relaxation problem.

\begin{example}
We consider Motzkin polynomial over the unit ball:
$$ 
\begin{array}{rl}
\min & f(x,y,z)=x^4y^2+x^2y^4-3x^2y^2z^2+z^6;\\
s.t & h (x,y,z)= 1-x^2-y^2-z^2  \geq 0 
\end{array}
$$
The polynomial $f$ is homogeneous and non negative on $\R^3$ but not a
sum of squares in $\R[x,y,z]$.\\

We add the projections of KKT ideal and we have the similar problem
$$\begin{array}{rl}
\min & x^4y^2+x^2y^4-3x^2y^2z^2+z^6;\\
s.t & -4xy^5+12xy^3z^2+4yx^5-12x^3yz^2=0;\\
&-4zx^4y-20zx^2y^3+12x^2yz^3+12yz^5=0;\\
&-20zx^3y^2-4zxy^4+12xy^2z^3+12xz^5=0;\\
&(4x^3y^2+2xy^4-6xy^2z^2)(-x^2-y^2-z^2+1)=0;\\
&(2x^4y+4x^2y^3-6x^2yz^2)(-x^2-y^2-z^2+1)=0;\\
&(-6x^2y^2z+6z^5)(-x^2-y^2-z^2+1)=0;
\end{array}
$$
where the first three equations are the $2\times 2$ minors of the
Jacobian matrix of $f$ and $h$ 
and the last three equations are the gradient ideal of $f$ multiplied by $h$.\\

In the first iteration the order is 5, the size of the Hankel matrix
$M_{\Lambda}^5$ is 156, $\min \Lambda(f) =0$, there is no duality gap.
We compute the kernel of this matrix:
$\ker M_{\Lambda}^5=\< z(y^2-z^2),x(y^2-z^2),z(x^2-z^2),y(x^2-z^2)
\>$. It generates the minimizer ideal
$I_{min}=(z(y^2-z^2),x(y^2-z^2),z(x^2-z^2),y(x^2-z^2))$ defining $6$
lines: $(y\pm z, x\pm z), (x,z), (y,z)$.
Here $\Vc(I_{min})$ is not included in $S$.
\end{example}

\begin{example}
We consider minimization of a linear function on a torus:
$$ \begin{array}{rl}
\min & f(x,y,z)= z\\
s.t & 9-10x^2-10y^2+6z^2+x^4+2x^2y^2+2x^2z^2+2y^2z^2+y^4+z^4=0
\end{array}$$
In the first iteration, the order is 2, the size of the Hankel matrix $M_{\Lambda}^2$ is $10$, $\min \Lambda(f) =-1$, there is no duality gap.
We compute the kernel of this matrix: $\ker
M_{\Lambda}^2=\< x^2+y^2-4,x(z+1),y(z+1), z (z+1), (z+1) \>$ which generates
the minimizer ideal $I_{min}=(x^2+y^2-4, z+1)$,
defining a circle which is the intersection of the torus with a tangent plane.
Notice that the multiplicity of this intersection has been removed in $I_{min}$.
\end{example}


\bibliographystyle{plain}

\end{document}